\documentclass[preprint,12pt]{elsarticle}

\usepackage{amssymb}
\usepackage{bm}
\usepackage{mathrsfs}
\usepackage[english]{babel}
\usepackage{amsmath, cases,amsfonts,mathrsfs}
\usepackage{epsfig, multirow, subfigure, txfonts}
\usepackage[small]{caption2}
\usepackage{color,hyperref}
\usepackage[thmmarks]{ntheorem}
{
\theoremstyle{nonumberplain}
\theoremheaderfont{\bfseries}
\theorembodyfont{\normalfont}
\theoremsymbol{\mbox{$\Box$}}
\newtheorem{proof}{Proof.}
}
\def\bc{\begin{center}}
\def\ec{\end{center}}

\def\bel{\begin{equation}}
\def\enl{\end{equation}}
\def\be{\begin{eqnarray*}}
\def\en{\end{eqnarray*}}
\def\i{{\bf i}}
\def\j{{\bf j}}
\def\k{{\bf k}}

\newtheorem{Theorem}{Theorem}[section]
 \newtheorem{Lem}{Lemma}[section]
\newtheorem{Cor}{Corollary}[section]
\newtheorem{Pro}{Proposition}[section]
\newtheorem{Def}{Definition}[section]
\newtheorem{Rem}{Remark}[section]
\newtheorem{Exam}{Example}[section]

\def\R{{\mathbb{R}}}
\def\C{{\mathbb{C}}}

\def\H{{\mathbb{H}}}
 
 \def\cF{{\cal F}}
\def\F{{\mathcal {F}}}
\def\L{{\mathcal {L}}}

\begin{document}
\begin{frontmatter}

\title{Hypercomplex Signal Energy Concentration
 in the Spatial and Quaternionic Linear Canonical Frequency Domains}

\author[firstaddress]{Cuiming~Zou}

\author[mymainaddress]{Kit Ian Kou\corref{mycorrespondingauthor}}
\cortext[mycorrespondingauthor]{Corresponding author}
\ead{kikou@umac.mo}

\address[firstaddress]{Department of Mathematics, Faculty of Science and Technology,
University of Macau, Taipa, Macao, China. Email: zoucuiming2006@163.com}
\address[mymainaddress]{Department of Mathematics, Faculty of Science and Technology,
University of Macau, Taipa, Macao, China.}

\begin{abstract}
Quaternionic Linear Canonical Transforms (QLCTs) are a family of integral transforms,
which generalized the quaternionic Fourier transform and quaternionic fractional Fourier transform.
In this paper, we extend the energy concentration problem for 2D hypercomplex signals (especially quaternionic signals).
The most energy concentrated signals both in 2D spatial and quaternionic linear canonical frequency domains simultaneously
are recently recognized to be the quaternionic prolate spheroidal wave functions (QPSWFs).
The improved definitions of QPSWFs are studied which gave reasonable properties.
The purpose of this paper is to understand the measurements of energy concentration
in the 2D spatial and quaternionic linear canonical  frequency domains.
Examples of  energy concentrated ratios between the truncated Gaussian function
and QPSWFs intuitively illustrate that QPSWFs are more energy concentrated signals.
\end{abstract}

\end{frontmatter}

\section{Introduction}
\label{S1}

The energy concentration problem in the time-frequency domain plays a crucial role in signal processing.
The foundation of this problem comes from 1960s the research group of bell labs \cite{HL2011}.
The problem states that for any given signal $f$ with its Fourier transform (FT)
\begin{eqnarray}\label{Eq.FT}
\F(f)(\omega):= \frac{1}{\sqrt{2\pi}}\int_{-\infty}^\infty f(t)e^{-\mathbf{i}\omega t}dt,
\end{eqnarray}
the energy ratios of the duration and bandwidth limiting of the signal $f$, i.e.,
 $\alpha_f:=\frac{\int_{-\tau}^\tau |f(t)|^2dt}{\int_{-\infty}^\infty |f(t)|^2dt}$ and
$\beta_f:=\frac{\int_{-\sigma}^\sigma |\F(f)(\omega)|^2d\omega}{\int_{-\infty}^\infty |\F(f)(\omega)|^2d\omega}$
of $f(t)$ both in fixed time $[-\tau,\tau]$ and frequency $[-\sigma,\sigma]$ domains, satisfy the following inequality
\begin{eqnarray}\label{Eq.bound}
\arccos\alpha_f+\arccos\beta_f\geq\arccos\sqrt{\lambda_0}.
\end{eqnarray}
Let $E_f:=\int_{-\infty}^\infty |f(t)|^2dt$ be the total energy of $f$.
By the Parseval theorem \cite{P1977}, the energy in time and frequency domains are equal, i.e., $E_f=E_{\F(f)}$.
Without loss of generality, we consider the unit energy signals throughout this paper, i.e., $E_f=1$.

The important constant $\lambda_0$ in Eq. (\ref{Eq.bound}) is the eigenvalue of the zero order prolate spheroidal wave functions (PSWFs).
The PSWFs are  originally   used to solve the Helmhotz equation in prolate spheroidal coordinates
by means of separation of variables  \cite{F1957, T1999}.
In 1960s, Slepian \emph{et al.} \cite{SP1961,LP1961,LP1962}
 found that PSWFs are solutions for the energy concentration problem of bandlimited signals \cite{P1977}.
Their real-valued PSWFs are solutions of the integral equation
\begin{eqnarray}\label{Eq.1DPSWFs}
\int_{-\tau}^{\tau}f(x)\frac{\sin \sigma (x-y)}{\pi (x-y)}dx=\alpha f(y),
\end{eqnarray}
where $\alpha$ are eigenvalues of PSWFs.
Here $[-\tau,\tau]$ and $[-\sigma,\sigma]$ are the fixed time and frequency domains, respectively.
Important properties of PSWFs are given in \cite{SP1961, S1964, S1976, K1992, M2013}.
The following properties follow form the general theory of integral equations and are stated without proof.
\begin{enumerate}
  \item
  Eq.(\ref{Eq.1DPSWFs}) has  solutions only for real, positive values eigenvalues $\alpha_n$.
These values is a monotonically decreasing sequence,
$ 1>\alpha_0>\alpha_1>...>\alpha_n>...> 0$, such that
$\lim_{n\rightarrow \infty}\alpha_n=0$.
  \item
  To each $\alpha_n$  there corresponds only one eigenfunction $\psi_{n}(x)$ with a constant factor.
The functions $\{\psi_{n}(x)\}_{n=0}^{\infty}$ form a real orthonormal set in $ \L^2([-\tau,\tau];\R)$.
  \item
  An arbitrary real $\sigma$-bandlimited function $f(x)$ can be written as a sum
$$f(x)=\Sigma_{n=0}^{\infty}a_n\psi_{n}(x),~~~\textrm{for}~ \textrm{all}~x\in \R,$$
where $a_n:=\int_{\R}f(x)\psi_{n}(x)dx$.
\end{enumerate}
These properties are useful in solving the energy concentration problem and other applications \cite{DP2005, ZWSW2014, WS2003, WS2004}.
Slepian \emph{et al.} \cite{S1964} naturally extended them to
higher dimension and discussed their approximation in some special case in the following years.
After that, the works on this functions are slowly developed until 1980s
a large number of engineering applied this functions to signal processing,
such as bandlimited signals  extrapolation, filter designing, reconstruction  and so on
\cite{LW1980,TP1987,MC2004}.

The PSWFs have received intensive attention in recent years.
There are many efforts to extend this kind of functions to various types of integral transformations.
Pei \emph{et al.} \cite{DP2005, PD2005} generalized PSWFs associated with the finite fractional Fourier transform (FrFT)
and applied to the sampling theory.
Zayed \emph{et al.} \cite{MZ2014, Z2014} generalized PSWFs  not only
associated with the finite FrFT but also associated with the linear canonical transforms (LCTs)
and applied to sampling theory.
Zhao \emph{et al.}  \cite{ZWSW2014, ZRMT2010} discussed the PSWFs associated with LCTs in detail and
presented the maximally concentrated sequence in both time and LCTs-frequency domains.
The wavelets based PSWFs constructed by Walter \emph{et al.} \cite{WS2003, WS2004, WS2005}
have some  desirable properties lacking in other wavelet systems.
Kou \emph{et al.} \cite{MKZ2013} developed the PSWFs with noncommutative structures in Clifford algebra.
They not only generalized the PSWFs in Clifford space (CPSWFs), but also extended the transform to Clifford LCT.
But they just gave some basic properties of this functions and
have not discussed details of the energy relationship for square integrable signals.
In this paper, we consider the energy concentration problem for hypercomplex signals,
especially for quaternionic  signals \cite{S1979,H2014b} associated with quaternionic LCTs (QLCTs) in detail.
The improvement definition of QPSWFs are considered for odd and even quaternionic signals.
The study is a great improvement on the one appeared in \cite{MKZ2013}.

The QLCT is a generalization of the quaternionic FT (QFT) and quaternionic FrFT (QFrFT).
The QFT and QFrFT are widely used for color image processing and
signal analysis in these years \cite{EBW1987, CTM2007, BLC2003, CKL2015}.
Therefore, it has more degrees of freedom than QFT and QFrFT,
the performance will be more advanced in color image processing.

In the present paper, we generalize the 1D PSWFs under the QLCTs to the quaternion space,
which are referred to as quaternionic PSWFs (QPSWFs).
The improved definition of QPSWFs associated with the QLCTs is studied and their some important properties are analyzed.
In order to find the relationship of $(\alpha_{\bm{f}},\beta_{\bm{f}})$ for any square integrable quaternionic signal,
we show that the Parseval theorem and studied the energy concentration problem associated with the QLCTs.
In particularly, we utilize the quaternion-valued functions multiply two special chirp signals
on both sides as a bridge between the QLCTs and the QFTs.
The main goal of the present study is to develop the energy concentration problem associated with QLCTs.
We find that the proposed QPSWFs  are the most energy concentrated quaternionic signals.

The body of the present paper is organized as follows.
In Section \ref{sec.QuaternionAlgebra} and \ref{sec.QLCT}, some basic facts of quaternionic algebra and QLCTs are given.
Moreover, the Parseval identity for quaternionic signals associated with the (two-sided) QLCTs are presented.
In Section \ref{S4}, the improved definition and some properties of QPSWFs associated with QLCTs are discussed.
The Section \ref{S5} presents the main results, it includes two parts.
In subsection \ref{S5.1}, we introduce the existence theorem for the maximum energy concentrated
bandlimited function on a fixed spatial domain associated with the QLCTs.
In subsection \ref{S5.2}, we discuss the energy extremal properties in fixed spatial
and QLCTs-frequency domains for any quaternionic signal.  In particular, we give an inequality to present the relationship
of energy ratios for any quaternionic signal, which is analogue to the high dimensional real signals.
Moreover, examples of  energy concentrated ratios between the truncated Gaussian function and
 QPSWFs are presented, which can intuitively illustrate that QPSWFs are the more energy concentrated signals.
Finally, some conclusion are drawn in Section \ref{S6}.

\section{Quaternionic Algebra}
\label{sec.QuaternionAlgebra}
The present section collects some basic facts about quaternions \cite{BDS1982, MGS2014},
which will be needed throughout the paper.

For all what follows, let $\H$ be the {\it Hamiltonian skew field of quaternions}:
\begin{eqnarray}\label{Eq.4partsHnumber}
\H := \{\bm{q}=q_0+{\bf i}q_1+{\bf j}q_2+{\bf k}q_3 \, | \, q_0,q_1,q_2,q_3\in\R\} ,
\end{eqnarray}
which is an associative non-commutative four-dimensional algebra.
The basis elements $\{\textbf{i}, \textbf{j}, \textbf{k}\}$ obey the Hamilton's multiplication rules:
\begin{eqnarray*}
\i^2=\j^2=\k^2=-1; \quad \; \i \j =-\j \i =\k, \; \j \k =-\k \j =\i,
\; \k \i=-\i \k =\j,
\end{eqnarray*}
and the usual component-wise defined addition.
 In this way the quaternionic algebra arises as a natural extension of the complex field $\mathbb{C}$.

The {\it quaternion conjugate} of a quaternion $\bm{q}$ is defined by
\begin{eqnarray*}
\overline{\bm{q}}:=q_0-{\bf i}q_1-{\bf j}q_2-{\bf k}q_3,\quad q_0,q_1,q_2,q_3\in\R.
\end{eqnarray*}
We write
$
{\bf Sc}(\bm{q}):=\frac{1}{2}(\bm{q}+\overline{\bm{q}})=q_0$ and ${\bf Vec} (\bm{q}):=\frac{1}{2}(q-\overline{q})={\bf i}q_1+{\bf j}q_2+{\bf k}q_3,$
which are the {\it scalar} and {\it vector parts} of $\bm{q}$, respectively. This leads to a norm of $\bm{q}\in\H$ defined by
$$
|\bm{q}| := \sqrt{\bm{q}\overline{\bm{q}}} = \sqrt{\overline{\bm{q}}\bm{q}} = (q_0^2+q_1^2+q_2^2+q_3^2)^{\frac{1}{2}}.
$$
Then we have
$
\overline{\bm{ p}\bm{q}}=\overline{\bm{q}}\cdot\overline{\bm{ p}}$, $|\bm{q}|=|\overline{\bm{q}}|$,
$|\bm{ pq}|=|\bm{ p}||\bm{q}|$, for any $\bm{ p}, \bm{q} \in\H$.
By (\ref{Eq.4partsHnumber}), a quaternion-valued function or, briefly,
an $\H$-valued function $\bm{f}:\R^2\to\H$ can be expressed in the following form:
$$
\bm{f}(x, y)=f_0(x, y)+{\bf i}f_1(x, y)+{\bf j}f_2(x, y)+{\bf k}f_3(x, y),\,
$$
where $f_i :\R^2 \rightarrow \R$ $(i=0, 1,2,3)$.
For convenience's sake, in the considerations to follow we will rewrite $\bm{f}$ in the following symmetric form \cite{H2007}:
\begin{eqnarray}\label{Eq.4partsHfunction}
\bm{f}(x, y)=f_0(x, y)+{\bf i}f_1(x, y)+f_2(x, y){\bf j}+{\bf i}f_3(x, y){\bf j}.
\end{eqnarray}
Properties (like integrability, continuity or differentiability) that are ascribed to $\bm{f}$ have to be fulfilled by all components $f_i$ $(i=0,1,2,3)$.

In order to state our results, we shall need some further notations.
The linear spaces $\mathcal{L}^p(\R^2;\H)$ ($1 \leq p <\infty$)
 consist of all \emph{$\H$-valued functions} in $\R^2$ under left multiplication
 by quaternions, whose $p$-th power is Lebesgue integrable in $\R^2$:
\begin{eqnarray*}
\mathcal{L}^p(\R^2;\H) := \left\{\bm{f} \; \Big| \; \bm{f}: \R^2 \to \H, \;
\|\bm{f}\|_{\mathcal{L}^p(\R^2;\H)} := \left( \int_{\R^2} |\bm{f}(x, y)|^p dxdy\right)^{1/p} < \infty \right\}.
\end{eqnarray*}
In this work, the {\it left} quaternionic inner product of $\bm{f}, \bm{ g} \in \mathcal{L}^2(\R^2; \H)$ is defined by
\begin{eqnarray} \label{Eq.Hinnerproduct}
<\bm{f},\bm{ g}>_{\mathcal{L}^2(\R^2;\H)} \, := \int_{\R^2}\bm{f}(x, y)\overline{\bm{ g}(x, y)}dxdy.
\end{eqnarray}
The reader should note that the norm induced by the inner product (\ref{Eq.Hinnerproduct}),
\begin{eqnarray*}
\|\bm{f}\|^2 =\|\bm{f}\|_{\mathcal{L}^2(\R^2; \H)}^2
:=<\bm{f},\bm{f}>_{\mathcal{L}^2(\R^2;\H)}= \int_{\R^2}|\bm{f}(x, y)|^2 dxdy.
\end{eqnarray*}
coincides with the $\mathcal{L}^2$-norm for $\bm{f}$, considered as a vector-valued function.

The angle between two non-zero functions $\bm{f}, \bm{ g} \in {\mathcal{L}^2(\R^2; \H)}$ is defined by
\begin{eqnarray}\label{Eq.angle}
\arg(\bm{f},\bm{ g}) :=\arccos \left(\frac{{\bf Sc}(< \bm{f},\bm{ g}>)}{\|\bm{f}\| \; \|\bm{ g}\|} \right).
\end{eqnarray}

The superimposed argument is well-defined since, obviously, it holds
\begin{eqnarray*}
|{\bf Sc}(< \bm{f},\bm{ g}>) |\leq |<\bm{f},\bm{ g}>_{\mathcal{L}^2(\Omega; \H)}| \leq
\|\bm{f}\| \; \|\bm{ g}\|.
\end{eqnarray*}


\section{ The Quaternionic Linear Canonical Transforms (QLCTs)}
\label{sec.QLCT}

The LCT was first proposed by Moshinsky and Collins \cite{C1970,MQ1971} in the 1970s.  It is a linear integral transform, which includes many special cases, such as the Fourier transform (FT), the FrFT, the Fresnel transform, the Lorentz transform and scaling operations. In a way, the LCT has more degrees of freedom and is more flexible than the FT and the FrFT, but with similar  computational costs  as the conventional FT. Due to the mentioned advantages, it is  of natural interest to extend the LCT to a quaternionic algebra framework. These extensions lead to the {\it Quaternionic Linear Canonical Transforms} (QLCTs). Due to the non-commutative property of multiplication of quaternions, there are different types of QLCTs. As explained in more detail below, we restrict our attention to the {\it two-sided} QLCTs \cite{KOM2016, XKM2016} of 2D quaternionic signals in this paper.
\subsection{Definition of QLCTs Revisited}

\begin{Def} [\textbf{Two-sided QLCTs}] \label{Def:QLCTs}
Let $A_i=\left(\begin{array}{ll} a_i & b_i\\c_i &
d_i\end{array}\right)\in \R^{2 \times 2}$ be a matrix parameter such that $\det (A_i)=1,$
for $i=1,2.$ The two-sided QLCTs of signals $\bm{f} \in
\mathcal{L}^1 \bigcap \mathcal{L}^2(\R^2;\H)$ are given by
\begin{eqnarray}\label{Eq.2sideQLCTs}
\mathcal{L} (\bm{f}) (u,v ):= \int_{\R^2} K^{\i}_{A_1}(x,u)
\bm{f}(x, y) K^{\j}_{A_2}(y,v)dxdy,
\end{eqnarray}
where the kernel functions are formulated by
\begin{eqnarray}\label{Eq.kernelQLCTs1}
K^{\i}_{A_1}(x,u):=\left\{ \begin{array}{ll} {1 \over \sqrt{\i 2 \pi b_1}} e^{\i
\left({a_1 \over 2 b_1} x^2-{1 \over b_1} xu +{d_1 \over 2
b_1} u^2 \right) }, & {\rm for } \;\, b_1 \not=0,\\[1.0ex]
\sqrt{d_1} e^{\i ({c_1 d_1 \over 2})
u^2}, & {\rm for } \;\, b_1 =0,\end{array}\right.
\end{eqnarray}
and
\begin{eqnarray}\label{Eq.kernelQLCTs2}
K^{\j}_{A_2}(y,v):=\left\{ \begin{array}{ll} {1 \over \sqrt{\j 2 \pi b_2}} e^{\j
\left({a_2 \over 2 b_2} y^2-{1 \over b_2} yv +{d_2 \over 2
b_2} v^2 \right) }, &  {\rm for } \;\, b_2 \not=0,\\[1.0ex]
\sqrt{d_2} e^{\j ({c_2 d_2 \over 2})
v^2}, & {\rm for } \;\, b_2 =0.\end{array}\right.
\end{eqnarray}
\end{Def}

It is significant to note that when $A_1=A_2=\left(\begin{array}{ll} 0 & 1\\-1 &
0\end{array}\right)$, the QLCT of $\bm{f}$ reduces to ${1 \over \sqrt{\i 2 \pi}} \cF (\bm{f}) (u,v) {1 \over \sqrt{\j 2 \pi}}$, where
\begin{eqnarray}\label{Eq.2sideQFT}
\cF (\bm{f}) (u,v )
:= \int_{\R^2} e^{-\i xu} \bm{f}(x, y) e^{-\j yv}
dxdy
\end{eqnarray}is the two-sided QFT of $\bm{f}$.
Note that when $b_i=0$ $(i=1,2)$, the QLCT of a signal is essentially
a chirp multiplication and is of no particular interest for our objective interests.
Without loss of generality, we set
$b_i >0$ $(i=1,2)$ throughout the paper.

\begin{Rem}
 Let $b_1, b_2 \not=0$. Using the Euler formula for the quaternionic linear canonical kernel
we can rewrite Eq. (\ref{Eq.2sideQLCTs}) in the following form:
\begin{eqnarray*}
\nonumber \mathcal{L} (f) (u,v)
= {-\i \sqrt{\i} \over 2 \pi \sqrt{b_1 b_2}}
\left(P_1+ \i P_2+P_3 \j +\i  P_4\j \right) (-\j \sqrt{\j}),
\end{eqnarray*}
where
\begin{eqnarray*}
P_1&:=&\int_{\R^2}  f(x, y)\cos \left({a_1 \over 2b_1} x^2-{1 \over b_1}xu +{d_1 \over 2b_1} u^2 \right)
\cos \left({a_2 \over 2b_2} y^2-{1 \over b_2} yv +{d_2 \over 2 b_2} v^2 \right)dxdy,\\
P_2&:=& \int_{\R^2} f(x, y)\sin \left({a_1 \over 2b_1} x^2-{1 \over b_1}xu +{d_1 \over 2b_1} u^2 \right)
\cos \left({a_2 \over 2b_2} y^2-{1 \over b_2} yv +{d_2 \over 2 b_2} v^2 \right)dxdy, \\
P_3&:=& \int_{\R^2} f(x, y)\cos \left({a_1 \over 2b_1} x^2-{1 \over b_1}xu +{d_1 \over 2b_1} u^2 \right)
\sin \left({a_2 \over 2b_2} y^2-{1 \over b_2} yv +{d_2 \over 2 b_2} v^2 \right) dxdy, \\
P_4&:=& \int_{\R^2} f(x, y) \sin \left({a_1 \over 2b_1} x^2-{1 \over b_1}xu +{d_1 \over 2b_1} u^2 \right)
  \sin \left({a_2 \over 2b_2} y^2-{1 \over b_2} yv +{d_2 \over 2 b_2} v^2 \right)  dxdy.
\end{eqnarray*}
 The above equation clearly shows how the QLCTs separate
{\bf real } signals $f(x,y)$ into four quaternionic components, i.e., the even-even,
odd-even, even-odd and odd-odd components of $f(x,y)$.
\end{Rem}

From Eq. (\ref{Eq.2sideQLCTs}) if $\bm{f} \in \mathcal{L}^1 \bigcap \mathcal{L}^2(\R^2;\H)$,
then the two-sided QLCTs $\mathcal{L}(\bm{f})(u,v)$ has a symmetric representation
\begin{eqnarray*}
\mathcal{L} (\bm{f})(u,v) = \mathcal{L} (f_0 )(u,v)+ \mathcal{L} (f_1 )(u,v)\mathbf{i} +
\mathcal{L} ( f_2 )(u,v) \mathbf{j}+ \mathbf{i}\mathcal{L} (f_3 )(u,v)\mathbf{j},
\end{eqnarray*}
where $\mathcal{L} (f_i )$ $(i=0,1,2,3)$ are the QLCTs of $f_i$ and they are $\H$-valued functions.

Under suitable conditions, the inversion of two-sided quaternionic linear canonical transforms of $\bm{f}(u,v)$ can be defined as follows.

\begin{Def} [\textbf{Inversion QLCTs}]
Suppose that $\bm{f} \in \mathcal{L}^1\bigcap \mathcal{L}^2 (\R^2, \H)$. Then the inversion of two-sided QLCTs of
$\bm{f}(u,v)$ are defined by
\begin{eqnarray}
\mathcal{L}^{-1}(\bm{f}) (x, y)
:=\int_{\R^2} K^{\i}_{A_1^{-1}}(x,u) \bm{f}(u,v) K^{\j}_{A_2^{-1}}(y,v) du dv,
\end{eqnarray}
where $A_i^{-1}=\left(\begin{array}{ll} d_i & -b_i\\ -c_i &
a_i\end{array}\right)$ and $\det (A_i^{-1})=1$ for $i=1,~2$.
\end{Def}

The following subsection describes the important relationship between QLCTs and QFT,
which will be used to establish the main results in Section \ref{S5}.
\subsection{The Relation Between QLCTs and QFT}

Note that the QLCTs of $\bm{f}$
multiple the chirp signals $2\pi \sqrt{b_1\mathbf{i}} e^{\mathbf{-i}\frac{d_1}{2b_1}u^2}$ on the left and $e^{\mathbf{-j}\frac{d_2}{2b_2}v^2}\sqrt{b_2\mathbf{j}}$ on the right
can be regarded as the QFT  on the scale domain.
Since
\begin{eqnarray}\label{Eq.LCTFTrelation}
&& 2\pi \sqrt{b_1\mathbf{i}} e^{\mathbf{-i}\frac{d_1}{2b_1}u^2}
\L(\bm{f}) (u,v)e^{\mathbf{-j}\frac{d_2}{2b_2}v^2}\sqrt{b_2\mathbf{j}}\\ \nonumber
&=&\int_{\R^2}
e^{\mathbf{-i} \frac{1}{b_1}xu}
\left(e^{\mathbf{i} \frac{a_1}{2b_1}x^2}\bm{f}(x,y)e^{\mathbf{j} \frac{a_2}{2b_2}y^2}\right)
e^{\mathbf{-j} \frac{1}{b_2}yv}
dxdy =\F(\tilde{\bm{f}}) \left(\frac{u}{b_1},\frac{v}{b_2}\right),
\end{eqnarray}
where $\tilde{\bm{f}}(x,y):=e^{\mathbf{i} \frac{a_1}{2b_1}x^2}\bm{f}(x,y)e^{\mathbf{j} \frac{a_2}{2b_2}y^2}$
is related to the parameter matrix $A_i,~i=1,2$ in Eq. (\ref{Eq.2sideQLCTs}).

\begin{Lem} [\textbf{Relation Between QLCT and QFT}]
Let $A_i=\left(\begin{array}{ll} a_i & b_i\\c_i &
d_i\end{array}\right)\in \R^{2 \times 2}$ be a real matrix parameter such that $\det (A_i)=1$
for $i=1,2.$ The relationship between two-sided QLCTs and QFTs of $\bm{f}\in
\mathcal{L}^1 \bigcap \mathcal{L}^2(\R^2;\H)$ are given by
\begin{eqnarray}\label{Eq.FLrelation}
\F (\tilde{\bm{f}}) \left(\frac{u}{b_1},\frac{v}{b_2}\right)
=2\pi \sqrt{b_1\mathbf{i}} e^{\mathbf{-i}\frac{d_1}{2b_1}u^2}
\L(\bm{f})(u,v)e^{\mathbf{-j}\frac{d_2}{2b_2}v^2}\sqrt{b_2\mathbf{j}},
\end{eqnarray} where
$\tilde{\bm{f}}(x,y)=e^{\mathbf{i} \frac{a_1}{2b_1}x^2}\bm{f}(x,y)e^{\mathbf{j} \frac{a_2}{2b_2}y^2}$.
\end{Lem}

\subsection{Energy Theorem Associated with QLCTs}
This subsection describes energy theorem of two-sided QLCTs  \cite{CK2016},
which will be applied to derive the extremal properties of QLCTs in Section \ref{S5}.

\begin{Theorem} [\textbf{Energy Theorem of the QLCTs}] \label{Th.PlancherelTheorem}
Any 2D $\H$-valued function  $\bm{f}\in L^2(\R^2, \H)$ and
its QLCT $\mathcal{L} (\bm{f})$ are  related by the Parseval identity
\begin{eqnarray}\label{Eq.ParsevalQLCTs}
\| \bm{f}\|^2= \|\mathcal{L} (\bm{f})\|^2.
\end{eqnarray}
\end{Theorem}
\begin{proof}
 For $\bm{f}\in L^2(\R^2, \H)$, direct computation shows that
\begin{eqnarray*}
||\mathcal{L}(\bm{f})||^2
= \int_{\R^2} \mathcal{L}(\bm{f})(u,v)\overline{\mathcal{L}(\bm{f})(u,v)}dudv
={\bf Sc}\left[\int_{\R^2}\mathcal{L}(\bm{f})(u,v)\overline{\mathcal{L}(\bm{f})(u,v)}dudv \right].
\end{eqnarray*}
Applying the definition of QLCTs, we have
\begin{eqnarray*}
||\mathcal{L}(\bm{f})||^2
&=&{\bf Sc}\left[\int_{\R^2}\left(\int_{\R^2} K^{\i}_{A_1}(x,u)\bm{f}(x, y) K^{\j}_{A_2}(y,v)dxdy\right)\overline{\mathcal{L}(\bm{f})(u,v)}dudv \right]\\
&=&{\bf Sc}\left[\int_{\R^4}K^{\i}_{A_1}(x,u)\bm{f}(x, y) K^{\j}_{A_2}(y,v)\overline{\mathcal{L}(\bm{f})(u,v)}dxdydudv \right]\\
&=&\int_{\R^4}{\bf Sc}\left[K^{\i}_{A_1}(x,u)\bm{f}(x, y) K^{\j}_{A_2}(y,v)\overline{\mathcal{L}(\bm{f})(u,v)}\right]dxdydudv.
\end{eqnarray*}
With ${\bf Sc}(\bm{qp})={\bf Sc}(\bm{pq})$  for any  $\bm{p},\bm{q} \in \H$  and
 $\overline{K^{\i}_{A_1}}=K^{\i}_{A_1^{-1}}$,
$\overline{K^{\j}_{A_2}}=K^{\j}_{A_2^{-1}}$, we have
\begin{eqnarray*}
||\mathcal{L}(\bm{f})||^2
&=&\int_{\R^4}{\bf Sc}\left[\bm{f}(x, y) K^{\j}_{A_2}(y,v)  \overline{\mathcal{L}(\bm{f})(u,v)}
 K^{\i}_{A_1}(x,u)\right] dxdydudv \\
&=&\int_{\R^4}{\bf Sc}\left[\bm{f}(x, y)  \overline{K^{\i}_{A^{-1}_1}(x,u)\mathcal{L}(\bm{f})(u,v) K^{\j}_{A^{-1}_2}(y,v) }
\right]  dxdydudv\\
&=&{\bf Sc}\left[ \int_{\R^2}\bm{f}(x, y)\overline{\int_{\R^2} K^{\i}_{A^{-1}_1}(x,u)\mathcal{L}(\bm{f})(u,v) K^{\j}_{A^{-1}_2}(y,v)dudv }dxdy\right]\\
&=&{\bf Sc}\left[ \int_{\R^2}\bm{f}(x, y)\overline{\bm{f}(x, y)}dxdy\right]
=\int_{\R^2}\bm{f}(x, y)\overline{\bm{f}(x, y)}dxdy=\| \bm{f}\|^2.
\end{eqnarray*}
Hence this completes the proof.
\end{proof}
Theorem \ref{Th.PlancherelTheorem} shows that the energy for an $\H$-valued signal
 in the spatial domain equals to the energy  in the QLCTs-frequency domain.
 The Parseval theorem allows the energy of an $\H$-valued signal to be considered on either the
spatial domain or the QLCTs-frequency domain, and exchange the domains
for convenience computation.

\begin{Cor}
The energy theorem of  $\bm{f}$ and
$\tilde{\bm{f}}(x,y)=e^{\mathbf{i} \frac{a_1}{2b_1}x^2}\bm{f}(x,y)e^{\mathbf{j} \frac{a_2}{2b_2}y^2}$
associated with their QFT is given by
\begin{eqnarray}\label{Eq.ParsevalQFT}
\| \bm{f}\|^2=\| \tilde{\bm{f}}\|^2= \|\mathcal{F} (\bm{f})\|^2.
\end{eqnarray}
\end{Cor}

\section{The Quaternionic Prolate Spheroidal Wave Functions}
\label{S4}

In the following, we first explicitly present the definition of PSWFs associated with QLCTs.

\subsection{Definitions of QPSWFs}
Consider the 1D PSWFs \cite{P1977, SP1961, S1964, MKZ2013},
let us extend the PSWFs to the quaternionic space associated with QLCTs.

\begin{Def} [\textbf{QPSWFs}]
The solutions of the following integral equation  in $ \L^1(\R^2;\H)$
\begin{eqnarray}\label{Eq.QPSWF}
\lambda_n \mathbf{i}^{n-\frac{1}{2}}\bm{\psi}_{n}(u,v)\mathbf{j}^{n-\frac{1}{2}}
:= \int_{\bm{\tau}}
K^{\mathbf{i}}_{A'_1}(x,u)\bm{\psi}_{n}(x,y)K^{\mathbf{j}}_{A'_2}(y,v)
dxdy,
\end{eqnarray}
are called the quaternionic prolate spheroidal wave functions (QPSWFs)
$\{\bm{\psi}_{n}(x,y)\}_{n=0}^\infty$ associated with QLCTs.
Here,  the complex valued $\lambda_n$ are the eigenvalues corresponding to the eigenfunctions $\bm{\psi}_{n}(x,y)$.
The real parameter matrix $A'_i:=\left(\begin{array}{cc} ca_i& b_i \\ cc_i&cd_i\end{array}\right)$
with $a_id_i-b_ic_i=1$, $b_i\neq 0$, for $i=1, 2$.
The real constant $c$ is a ratio about the frequency domain
 $\bm{\sigma}:=[-\sigma,\sigma]\times[-\sigma,\sigma]$
and the spatial domain $\bm{\tau}:=[-\tau,\tau]\times[-\tau,\tau]$,
where $c:=\frac{\sigma}{\tau},~0<c<\infty$.
Eq. (\ref{Eq.QPSWF}) is named the {\it finite QLCTs form of QPSWFs}.
\end{Def}

Note that for simplicity of presentation, we write
$\int_{-\tau}^{\tau} \int_{-\tau}^{\tau}=\int_{\bm{\tau}}$
and  $\int_{-\sigma}^{\sigma} \int_{-\sigma}^{\sigma}=\int_{\bm{\sigma}}$.

\begin{Rem}
The solutions of this integral equation in Eq. (\ref{Eq.QPSWF})
are well established in some special cases.

\begin{description}
\item[(i)] In the square region $\bm{\tau}=[-\tau,\tau]\times[-\tau,\tau]$,
if QLCTs are degenerated to 2D Fourier transform (FT),
then QPSWFs becomes the 2D real PSWFs, which is given by
\begin{eqnarray*}
\lambda_n  \psi_{n}(u,v) = \int_{\bm{\tau}}e^{\mathbf{i}cxu} \psi_{n}(x,y)e^{\mathbf{i}cyv}dxdy.
\end{eqnarray*}
Here, if $\psi_{n}(x,y)$ is separable, i.e., $\psi_{n}(x,y)=\psi_{n}(x)\psi_{n}(y)$, then
the 2D PSWFs can be regarded as the product of two 1D PSWFs.
To aid the reader, see \cite{SP1961} for more complete accounts of this subject.

\item[(ii)] In  a unit disk, the QLCTs are degenerated to 2D FT, then the QPSWFs between the circular PSWFs \cite{S1979}
\begin{eqnarray*}
\lambda_n \psi_{n}(u,v)=\int_{x^2+y^2\leq 1}
e^{\mathbf{i}(cxu+cyv)}\psi_{n}(x,y)dxdy.
\end{eqnarray*}
\end{description}
\end{Rem}
\begin{Rem}
We call the right-hand side of Eq. (\ref{Eq.QPSWF}) is the finite QLCTs. However, $\det(A'_i)=1,~i=1,2$ only for the $c=1$.
There is a scale factor $c$ added to the parameter matrix, which is different from the definition of QLCTs.
\end{Rem}

\subsection{Properties of QPSWFs}

Some important properties of QPSWFs will be considered in this part,
which are crucial in solving the energy concentration problem.

\begin{Pro}[\textbf{Low-pass Filtering Form in $\bm{\tau}$}]
Let $\bm{\psi}_{n}(x,y)\in\L^1(\R^2;\H)$ be the QPSWFs associated with their QLCTs and
$ \tilde{\bm{\psi}}(x,y):=e^{\mathbf{i} \frac{ca_1}{2b_1}x^2}\bm{\psi}_{n}(x,y)e^{\mathbf{j} \frac{ca_2}{2b_2}y^2}$.
Then $\{\tilde{\bm{\psi}}_n(x,y)\}_{n=0}^\infty$ are solutions of the following integral equation
\begin{eqnarray}\label{Eq.lowpass}
\mu_n \tilde{\bm{\psi}}_{n}(u,v)=
\int_{\bm{\tau}}
\tilde{\bm{\psi}}_{n}(x,y)
\frac{\sin \sigma (x-u)}{\pi (x-u)}
\frac{\sin \sigma (y-v)}{\pi (y-v)}
dxdy,
\end{eqnarray}
where $\mu_n:=c^4b_1b_2\lambda_n^2$  for  $n=0,1,2\cdots$ are the eigenvalues corresponding to $\tilde{\bm{\psi}}_{n}(x,y)$
and $a_id_i-b_ic_i=1$, $b_i\neq 0$, for $i=1, 2$, and $c:=\frac{\sigma}{\tau}, 0<c<\infty$.
Eq. (\ref{Eq.lowpass}) is named the {\it low-pass filtering form of QPSWFs associated with QLCTs}.
\end{Pro}

\begin{proof}
We shall show that Eq. (\ref{Eq.lowpass}) is derived by the Eq. (\ref{Eq.QPSWF}).
Straightforward computations of the right-hand side of Eq. (\ref{Eq.lowpass}) show that
\begin{eqnarray*}
&&\int_{\bm{\tau}}\tilde{\bm{\psi}}_{n}(x,y)
\frac{\sin \sigma (x-u)}{\pi (x-u)}
\frac{\sin \sigma (y-v)}{\pi (y-v)}dxdy\\
&=&
\int_{\bm{\tau}}
\frac{\sin \sigma (u-x)}{\pi (u-x)}e^{\mathbf{i}
\frac{ca_1}{2b_1}x^2}\bm{\psi}_{n}(x,y)e^{\mathbf{j} \frac{ca_2}{2b_2}y^2}
\frac{\sin \sigma (v-y)}{\pi (v-y)}
dxdy.
\end{eqnarray*}
Applying the following two important equations \cite{A2003} to the last integral,
\begin{eqnarray}\label{Eq.FTofSinc}
 \frac{1}{2\pi}\int_{-\sigma}^{\sigma}e^{\mathbf{i}x u}dx
=\frac{\sin (\sigma u)}{\pi u} ~~~ \mathrm{and} ~~~
\frac{1}{2\pi}\int_{-\sigma}^{\sigma}e^{\mathbf{j}y v}dy
=\frac{\sin (\sigma v)}{\pi v},
\end{eqnarray}
then we have
\begin{eqnarray*}
&&\int_{\bm{\tau}}
\frac{\sin \sigma (u-x)}{\pi (u-x)}e^{\mathbf{i}
\frac{ca_1}{2b_1}x^2}\bm{\psi}_{n}(x,y)e^{\mathbf{j} \frac{ca_2}{2b_2}y^2}
\frac{\sin \sigma (v-y)}{\pi (v-y)}dxdy\\
&=&
\frac{1}{(2\pi)^2}\int_{\bm{\tau}} \int_{\bm{\sigma}}
e^{\mathbf{i}v_1(u-x)}e^{\mathbf{i} \frac{ca_1}{2b_1}x^2}
\bm{\psi}_{n}(x,y)e^{\mathbf{j} \frac{ca_2}{2b_2}y^2}e^{\mathbf{j}v_2(v-y)}
dv_1dv_2dxdy\\
&=&
\frac{1}{(2\pi)^2}\int_{\bm{\sigma}}e^{\mathbf{i}v_1u}
\left[\int_{\bm{\tau}}
e^{-\mathbf{i} \frac{cx}{b_1}(\frac{b_1v_1}{c})}e^{\mathbf{i} \frac{ca_1}{2b_1}x^2}
\bm{\psi}_{n}(x,y)
e^{\mathbf{j} \frac{ca_2}{2b_2}y^2}e^{-\mathbf{j} \frac{cy}{b_2}(\frac{b_2v_2}{c})}dxdy
\right]e^{\mathbf{j}v_2v}dv_1dv_2.
\end{eqnarray*}
Combining Eq.(\ref{Eq.QPSWF}) with the parameter matrices
$A'_i=\left(\begin{array}{cc} ca_i& b_i \\
cc_i&cd_i\end{array}\right), i=1,2$,
and
$A'_i=\left(\begin{array}{cc} -cd_ib_i^2&b_i\\
cc_i&-\frac{ca_i}{b_i^2}\end{array}\right), i=1,2$,
we have
\begin{eqnarray*}
&&\frac{1}{(2\pi)^2}\int_{\bm{\sigma}}e^{\mathbf{i}v_1u}
\left[\int_{\bm{\tau}}
e^{-\mathbf{i} \frac{cx}{b_1}(\frac{b_1v_1}{c})}e^{\mathbf{i} \frac{ca_1}{2b_1}x^2}
\bm{\psi}_{n}(x,y)
e^{\mathbf{j} \frac{ca_2}{2b_2}y^2}e^{-\mathbf{j} \frac{cy}{b_2}(\frac{b_2v_2}{c})}dxdy
\right]e^{\mathbf{j}v_2v}dv_1dv_2\\
&=&
\frac{1}{2\pi}\int_{\bm{\sigma}}e^{\mathbf{i}v_1u}
\left[
\lambda_n
e^{\mathbf{-i}\frac{cd_1}{2b_1}(\frac{b_1v_1}{c})^2}
\mathbf{i}^{n} \sqrt{cb_1}
\bm{\psi}_{n}\Big(\frac{b_1v_1}{cb_1},\frac{b_2v_2}{cb_2}\Big)
\sqrt{cb_2} \mathbf{j}^{n}
e^{\mathbf{-j}\frac{cd_2}{2b_2}(\frac{b_2v_2}{c})^2}
\right]
e^{\mathbf{j}v_2v}dv_1dv_2\\
&=&\frac{1}{2\pi}
\lambda_n \int_{\bm{\sigma}}e^{\mathbf{i}v_1u}
\left[e^{\mathbf{-i}\frac{cd_1b_1}{2}(\frac{v_1}{c})^2}
\mathbf{i}^{n} \sqrt{cb_1}
\bm{\psi}_{n}\Big(\frac{v_1}{c},\frac{v_2}{c}\Big)
\sqrt{cb_2} \mathbf{j}^{n}
e^{\mathbf{-j}
\frac{cd_2b_2}{2}(\frac{v_2}{c})^2}
\right]
e^{\mathbf{j}v_2v}dv_1dv_2\\
&=&\frac{1}{2\pi}
\lambda_n c^3\sqrt{b_1b_2}\mathbf{i}^{n}
\left[\int_{\bm{\tau}}
e^{\mathbf{i}w_1cu}
e^{\mathbf{-i}\frac{cd_1b_1}{2}w_1^2}
\bm{\psi}_{n}(w_1,w_2)
e^{\mathbf{-j}\frac{cd_2b_2}{2}w_2^2}
e^{\mathbf{j}w_2cv}dw_1dw_2
\right]
\mathbf{j}^{n}\\
&=&\lambda_n c^3\sqrt{b_1b_2}\mathbf{i}^{n}
\left[
\lambda_n e^{-\mathbf{i}\frac{c(\frac{-a_1}{{b_1}^2})}{2b_1}(ub_1)^2}
\frac{(-\mathbf{i})^n}{\sqrt{\mathbf{i}}} \sqrt{cb_1\mathbf{i}}
\bm{\psi}_{n}\Big(\frac{ub_1}{b_1},\frac{vb_2}{b_2}\Big)
\sqrt{cb_2\mathbf{j}} \frac{(-\mathbf{j})^n}{\sqrt{\mathbf{j}}}
e^{-\mathbf{j}\frac{c(\frac{-a_2}{{b_2}^2})}{2b_2}(vb_2)^2}
\right]
\mathbf{j}^{n}\\
&=&
\lambda_n^2 c^4b_1b_2e^{\mathbf{i}\frac{ca_1}{2b_1}u^2}
\bm{\psi}_{n}(u,v)e^{\mathbf{j}\frac{ca_2}{2b_2}v^2}
=c^4b_1b_2\lambda_n^2 \tilde{\bm{\psi}}_{n}(u,v)=:\mu_n \tilde{\bm{\psi}}_{n}(u,v).
\end{eqnarray*}
The proof is complete.
\end{proof}

\begin{Rem}
For the specific parameters $a_i=0, b_i=1, c_i=-1, d_i=0, i=1,2$, Eq. (\ref{Eq.lowpass}) becomes the
low-pass form of QPSWFs associated with QFT
 \begin{eqnarray*}
\int_{\bm{\tau}}\bm{\psi}_{n}(x,y) \frac{\sin \sigma (x-u)}{\pi (x-u)}\frac{\sin \sigma (y-v)}{\pi (y-v)}dxdy
=c^2\lambda_n^2 \bm{\psi}_{n}(u,v).
\end{eqnarray*}
\end{Rem}

To obtain the following property,
we shall show a special convolution theorem of any $\H$-valued signal and real-valued signal.

\begin{Lem}\label{Lem.convolution}
Let $\bm{f}\in \L^1(\R^2;\H)$ and $g\in\L^1(\R^2;\R)$ associated with their QFT $\F(\bm{f})$ and $\F(g)$
with $\bm{f}=f_0+\mathbf{i}f_1+f_2\mathbf{j}+\mathbf{i}f_3\mathbf{j}$,
where $f_i\in \L^1(\R^2;\R)$, $i=0,1,2,3$ and $\F(g)\in\L^1(\R^2;\R)$.
The convolution of $\bm{f}$ and $g$ is defined as
\begin{eqnarray}
(\bm{f}\ast g) (s,t):=\int_{\R^2}\bm{f}(x,y)g(s-x,t-y)dxdy.
\end{eqnarray}
Then the QFT for $\bm{f}\ast g$ holds
\begin{eqnarray}\label{Eq.convolution}
&&\F\left(\bm{f}\ast g\right)(u,v)\\
&=&\F(f_0+\mathbf{i}f_1)(u,v)\F(g)(u,v)+\F(f_2\mathbf{j}+\mathbf{i}f_3\mathbf{j})(u,v)\F(g)(-u,v). \nonumber
\end{eqnarray}
\end{Lem}
\begin{proof}
Let $s-x=m$ and $t-y=n$,
straightforward computation the QFT in Eq.(\ref{Eq.2sideQFT}) of the convolution between
$f$ and $g$ shows that
\begin{eqnarray*}
&&\F\left(\int_{\R^2}\bm{f}(x,y)g(s-x,t-y)dxdy\right)(u,v)\\
&=&\int_{\R^2}e^{-\mathbf{i}su}\left(\int_{\R^2}\bm{f}(x,y)g(s-x,t-y)dxdy\right)e^{-\mathbf{j}tv}dsdt\\
&=&\int_{\R^2}e^{-\mathbf{i}(x+m)u}\left(\int_{\R^2}\bm{f}(x,y)g(m,n)dxdy\right) e^{-\mathbf{j}(y+n)v}dmdn\\
&=&\int_{\R^4}e^{-\mathbf{i}mu}e^{-\mathbf{i}xu}
\left[\left(f_0(x,y)+\mathbf{i}f_1(x,y)\right)+\left(f_2(x,y)\mathbf{j}+\mathbf{i}f_3(x,y)\mathbf{j}\right)\right]
g(m,n) e^{-\mathbf{j}yv}e^{-\mathbf{j}nv}dxdydmdn.
\end{eqnarray*}
With $e^{-\mathbf{i}}\mathbf{j}=\mathbf{j}e^{\mathbf{i}}$, the last integral becomes
\begin{eqnarray*}
&&\int_{\R^4}e^{-\mathbf{i}xu}
\left(f_0(x,y)+\mathbf{i}f_1(x,y)\right)e^{-\mathbf{i}mu}g(m,n)e^{-\mathbf{j}nv}e^{-\mathbf{j}yv}dxdydmdn\\
+&&\int_{\R^4}e^{-\mathbf{i}xu}
\left(f_2(x,y)\mathbf{j}+\mathbf{i}f_3(x,y)\mathbf{j}\right)e^{\mathbf{i}mu}g(m,n)e^{-\mathbf{j}nv}e^{-\mathbf{j}yv}dxdydmdn\\
=&&\int_{\R^2}e^{-\mathbf{i}xu}
\left(f_0(x,y)+\mathbf{i}f_1(x,y)\right)\F(g)(u,v)e^{-\mathbf{j}yv}dxdy\\
+&&\int_{\R^2}e^{-\mathbf{i}xu}
\left(f_2(x,y)\mathbf{j}+\mathbf{i}f_3(x,y)\mathbf{j}\right)\F(g)(-u,v)e^{-\mathbf{j}yv}dxdy.
\end{eqnarray*}
Since we have known $\F(g)$ is real-valued, then we have
\begin{eqnarray*}
&&\int_{\R^2}e^{-\mathbf{i}xu}
\left(f_0(x,y)+\mathbf{i}f_1(x,y)\right)e^{-\mathbf{j}yv}\F(g)(u,v)dxdy\\
+&&\int_{\R^2}e^{-\mathbf{i}xu}
\left(f_2(x,y)\mathbf{j}+\mathbf{i}f_3(x,y)\mathbf{j}\right)e^{-\mathbf{j}yv}\F(g)(-u,v)dxdy\\
=&&\F(f_0)(u,v)\F(g)(u,v)+\mathbf{i}\F(f_1)(u,v)\F(g)(u,v)\\
+&&
\F(f_2)(u,v)\F(g)(-u,v)\mathbf{j}+\mathbf{i}\F(f_3\mathbf{j})(u,v)\F(g)(-u,v).
\end{eqnarray*}
This completes the proof.
\end{proof}

Note that if the real signal $g(x,y)=g(-x,y)$ with $\F(g)$ is real valued, then $\F(g)(u,v)=\F(g)(-u,v)$.
It means that
\begin{eqnarray*}
&&\F\left(\bm{f}\ast g\right)(u,v)\\
&=&\F(f_0+\mathbf{i}f_1)(u,v)\F(g)(u,v)+\F(f_2\mathbf{j}+\mathbf{i}f_3\mathbf{j})(u,v)\F(g)(-u,v)\\
&=&\F(f_0+\mathbf{i}f_1+f^2\mathbf{j}+\mathbf{i}f_3\mathbf{j})(u,v)\F(g)(u,v)\\
&=&\F(\bm{f})(u,v)\F(g)(u,v).
\end{eqnarray*}

\begin{Rem}
The convolution theorems for quaternion Fourier transform was given in \cite{MRR2013}.
Lemma \ref{Lem.convolution} is following the idea of Theorem 13 and Lemma 14 in \cite{MRR2013}.
For completeness, we proof the convolution formula in Eq. (\ref{Eq.convolution}).
\end{Rem}

\begin{Pro}
Let $\bm{\psi}_{n}(x,y)\in\L^1(\R^2;\H)$ be the QPSWFs associated with QLCTs and
$ \tilde{\bm{\psi}}_n(x,y)=e^{\mathbf{i} \frac{ca_1}{2b_1}x^2}\bm{\psi}_{n}(x,y)e^{\mathbf{j} \frac{ca_2}{2b_2}y^2}$,
$\{\tilde{\bm{\psi}}_n(x,y)\}_{n=0}^\infty$ satisfies
\begin{eqnarray}\label{Eq.allpass}
\tilde{\bm{\psi}}_{n}(u,v)
=\int_{\R^2}
\tilde{\bm{\psi}}_{n}(x,y)
\frac{\sin \sigma (x-u)}{\pi (x-u)}
\frac{\sin \sigma (y-v)}{\pi (y-v)}
dxdy,
\end{eqnarray}
which extends the integral of $\tilde{\bm{\psi}}_{n}(x,y)$ from $\bm{\tau}$ to $\R^2$.
\end{Pro}

\begin{proof}
Let
$
p_{\bm{\tau}}(x,y):= \left\{
\begin{array}{ll}
1, & (x,y) \in \bm{\tau},\\
0, &otherwise,
\end{array}\right.
$
Eq. (\ref{Eq.lowpass})  is actually a convolution of
$p_{\bm{\tau}}(x,y)\tilde{\bm{\psi}}_{n}(x,y)$ with two-dimensional \emph{sinc} kernel $\frac{\sin (\sigma x)}{\pi x}\frac{\sin (\sigma y)}{\pi y}$  as follows
\begin{eqnarray}\label{Eq.convl1}
\mu_n \tilde{\bm{\psi}}_{n}(u,v)
=\int_{\R^2}p_{\bm{\tau}}(x,y)\tilde{\bm{\psi}}_{n}(x,y)
\frac{\sin \sigma (x-u)}{\pi (x-u)}\frac{\sin \sigma (y-v)}{\pi (y-v)}dxdy.
\end{eqnarray}
Denote $\phi(x,y):=p_{\bm{\tau}}(x,y)\tilde{\bm{\psi}}_{n}(x,y)$.
From Lemma \ref{Lem.convolution}, let $g(x,y)=\frac{\sin \sigma x}{\pi x}\frac{\sin \sigma y}{\pi y}$,
the $g(x,y)=g(-x,y)$ and its QFT $p_{\bm{\sigma}}(u,v)$ is real valued function.
Then taking QFT to the both sides of Eq. (\ref{Eq.convl1}), we have
\begin{eqnarray}\label{Eq.QFTconvl1}
\mu_n\F(\tilde{\bm{\psi}}_{n})(u',v')
=\F(\phi)(u',v')p_{\bm{\sigma}}(u',v').
\end{eqnarray}
Immediately, we obtain that  $\F(\tilde{\bm{\psi}}_{n})(u',v')=0$ for $| u'|>\sigma~and~| v'|>\sigma$,
i.e.,
\begin{eqnarray}\label{Eq.FfPsigma}
\F(\tilde{\bm{\psi}}_{n})(u',v')
=\F(\tilde{\bm{\psi}}_{n})(u',v')p_{\bm{\sigma}}(u',v').
\end{eqnarray}
Here
$
p_{\bm{\sigma}}(x,y):= \left\{
\begin{array}{ll}
1, & (x,y) \in \bm{\sigma},\\
0, &otherwise.
\end{array}\right.
$
From Lemma \ref{Lem.convolution}, taking the inverse  QFT on both sides of the above equation,
it follows that $\tilde{\bm{\psi}}_{n}$ satisfies Eq. (\ref{Eq.allpass}),
which extends the integral domain of $\tilde{\bm{\psi}}_{n}(x,y)$ from $\bm{\tau}$ to $\R^2$.
\end{proof}

The Propositions \ref{Pro.Eigenvalues} and \ref{Pro.Orthogonalintau}
follow from the general theory of integral equations of Hermitian kernel
and are stated without proof \cite{K1992, M2013}.
\begin{Pro}[\textbf{Eigenvalues}]\label{Pro.Eigenvalues}
 Eq. (\ref{Eq.lowpass}) has  solutions for real or complex $\mu_n$.
These values are a monotonically decreasing sequence,
$ 1>|\mu_0|>|\mu_1|>...>|\mu_n|>...$, and satisfy
$\lim_{n\rightarrow \infty}|\mu_n| =0$.
\end{Pro}

\begin{Pro}[\textbf{Orthogonal in $\bm{\tau}$}]\label{Pro.Orthogonalintau}
For different eigenvalues $\{\mu_n\}_{n=0}^{\infty}$, the corresponding eigenfunctions $\{\tilde{\bm{\psi}}_{n}(x,y)\}_{n=0}^{\infty}$
are an orthonormal set in $\bm{\tau}$, i.e.,
\begin{eqnarray}\label{Eq.QFMtau}
\int_{\bm{\tau}}\tilde{\bm{\psi}}_{n}(x,y)\overline{\bm{\tilde{\psi}}_{m}(x,y)}
dxdy= \left\{
\begin{array}{ll}
\mu_n, & m=n,\\
0,  &otherwise.
\end{array}\right.
\end{eqnarray}
\end{Pro}

\begin{Pro}[\textbf{Orthogonal in $\R^2$}]
The eigenfunctions $\{ \tilde{\bm{\psi}}_{n}(x,y)\}_{n=0}^{\infty}$ form an orthonormal system in $\R^2$, i.e.,
\begin{eqnarray}\label{Eq.QFMR2orth}
\int_{\R^2}\tilde{\bm{\psi}}_{n}(x,y)\overline{\bm{\tilde{\psi}}_{m}(x,y)}dxdy
= \left\{
\begin{array}{ll}
1, & m=n,\\
0,  &otherwise.
\end{array}\right.
\end{eqnarray}
\end{Pro}

\begin{proof}
Combining Eq. (\ref{Eq.lowpass}), the orthogonality in $\R^2$ can be immediately deduced as follows
\begin{eqnarray*}
 &&\int_{\R^2}\tilde{\bm{\psi}}_{n}(x,y)\overline{\tilde{\bm{\psi}}_{m}(x,y)}dxdy\\
&=&\int_{\R^2}
\left(\frac{1}{\mu_n}\int_{\bm{\tau}}\tilde{\bm{\psi}}_{n}(s,t)
\frac{\sin \sigma (x-s)}{\pi (x-s)}\frac{\sin \sigma (y-t)}{\pi (y-t)}dsdt\right)\\
&&~~~~~~~
\overline{\left(\frac{1}{\mu_m}\int_{\bm{\tau}}\tilde{\bm{\psi}}_{m}(u,v)
\frac{\sin \sigma (x-u)}{\pi (x-u)}\frac{\sin \sigma (y-v)}{\pi (y-v)}dudv\right)}dxdy\\
&=&\frac{1}{\mu_n\mu_m}\int_{\bm{\tau}}\int_{\bm{\tau}}
\tilde{\bm{\psi}}_{n}(s,t)\overline{\tilde{\bm{\psi}}_{m}(u,v)}dsdtdudv\\
&&
\left(\int_{\R^2}
 \frac{\sin \sigma (x-s)}{\pi (x-s)}\frac{\sin \sigma (y-t)}{\pi (y-t)}
 \frac{\sin \sigma (x-u)}{\pi (x-u)}\frac{\sin \sigma (y-v)}{\pi (y-v)}dxdy\right) \\
 &=&\frac{1}{\mu_n\mu_m} \int_{\bm{\tau}}\left(\int_{\bm{\tau}}
\tilde{\bm{\psi}}_{n}(s,t)
\frac{\sin \sigma (s-u)}{\pi (s-u)}\frac{\sin \sigma (t-v)}{\pi (t-v)}dsdt\right)
\overline{\tilde{\bm{\psi}}_{m}(u,v)}dudv\\
&=&\frac{1}{\mu_m}\int_{\bm{\tau}}\tilde{\bm{\psi}}_{n}(u,v)\overline{ \tilde{\bm{\psi}}_{m}(u,v)}dudv
=\left\{
\begin{array}{ll}
1, & m=n,\\
0,  &otherwise.
\end{array}\right.
\end{eqnarray*}

\end{proof}

\section{Main Results}
\label{S5}
In the present section, we will consider the energy concentration problem of bandlimited
$\H$-valued signals in fixed spatial and QLCTs-frequency domains.
The definitions and notations of bandlimited $\H$-valued signals associated with QLCTs and QFT are introduced in the following.

\begin{Def}[\textbf{$\bm{\sigma}$-bandlimited $\H$-valued signal associated with QLCTs}]

an $\H$-valued signal $\bm{f}(x,y)$ with finite energy is $\bm{\sigma}$-bandlimited  associated with QLCTs,
 if its QLCTs vanishes for all $(u,v)$ outside the region $\bm{\sigma}$, i.e.,
 \begin{eqnarray}
 \L(\bm{f})(u,v)=0, ~~~\mathrm{for}~~ (u,v)\in  \R^2\setminus\bm{\sigma}.
\end{eqnarray}
Denote $\mathscr{B}_{\bm{\sigma}}$ the set of $\bm{\sigma}$-bandlimited $\H$-valued signals  associated with QLCTs, i.e.,
 \begin{eqnarray}
 \mathscr{B}_{\bm{\sigma}}:=\left\{f\in \mathcal{L}^2(\R^2;\H)\Big| \textrm{supp}(\L(\bm{f}(u,v)))\in \bm{\sigma}\right\}.
\end{eqnarray}
\end{Def}

\begin{Def}[\textbf{$\bm{\sigma}$-bandlimited $\H$-valued signal  associated with QFT}]
an $\H$-valued signal $\bm{f}(x,y)$ with finite energy is $\bm{\sigma}$-bandlimited  associated with QFT,
 if its QFT vanishes for all $(u,v)$ outside the region $\bm{\sigma}$, i.e.,
 \begin{eqnarray}
 \F(\bm{f})(u,v)=0, ~~~\mathrm{for}~~ (u,v)\in  \R^2\setminus\bm{\sigma}.
\end{eqnarray}
Denote $\tilde{\mathscr{B}}_{\bm{\sigma}}$ the set of the $\bm{\sigma}$-bandlimited $\H$-valued signals  associated with QFT, i.e.,
\begin{eqnarray}
\tilde{\mathscr{B}}_{\bm{\sigma}}:=\left\{f\in \mathcal{L}^2(\R^2;\H)\Big| \textrm{supp}(\F(\bm{f}(u,v)))\in \bm{\sigma}\right\}.
\end{eqnarray}
\end{Def}

Note that the relationship between QLCT and QFT for an $\H$-valued signal $\bm{f}$
\begin{eqnarray*}
\F (\tilde{\bm{f}}) \left(\frac{u}{b_1},\frac{v}{b_2}\right)
=2\pi \sqrt{b_1\mathbf{i}} e^{\mathbf{-i}\frac{d_1}{2b_1}u^2}
\L(\bm{f})(u,v)e^{\mathbf{-j}\frac{d_2}{2b_2}v^2}\sqrt{b_2\mathbf{j}}.
\end{eqnarray*}
That is to say if $\bm{f}\in \mathscr{B}_{\bm{\sigma}}$, then for the
 $\F (\tilde{\bm{f}}) \left(\frac{u}{b_1},\frac{v}{b_2}\right)$, $(u,v)$ is also in $\bm{\sigma}$, that means
 \begin{eqnarray}
(\frac{u}{b_1},\frac{v}{b_2})\in [\frac{-\sigma}{b_1},\frac{\sigma}{b_1}]\times [\frac{-\sigma}{b_2},\frac{\sigma}{b_2}]=:
 \tilde{\bm{\sigma}}.
 \end{eqnarray}
Then $\tilde{\bm{f}}\in \tilde{\mathscr{B}}_{\tilde{\bm{\sigma}}}$, because
\begin{eqnarray}
\F (\tilde{\bm{f}}) (u,v)=0, ~~~\mathrm{for}~~
 (u,v)\in  \R^2\setminus\tilde{\bm{\sigma}}.
\end{eqnarray}

Now we  pay attention to  the energy concentration problem  associated with QLCTs.
To be specific, the energy concentration problem  associated with QLCTs  aims to obtain the relationship of the following two energy ratios for
any $\H$-valued signal $\bm{f}$ with finite energy in a fixed spatial and QLCTs-frequency domains, i.e., $\bm{\tau}$ and $\bm{\sigma}$,
\begin{eqnarray}\label{Eq.extremal}
\alpha_{\bm{f}} :=\frac{\parallel p_{\bm{\tau}}\bm{f} \parallel^2}{\parallel \bm{f} \parallel^2}
~~~\mathrm{and}~~~
\beta_{\bm{f}} :=\frac{\parallel p_{\bm{\sigma}}\L(\bm{f}) \parallel^2}{\parallel \L(\bm{f}) \parallel^2}.
\end{eqnarray}
By the Parseval identity in Eq. (\ref{Eq.ParsevalQFT}),
the two ratios can also be obtained by
\begin{eqnarray}
\alpha_{\bm{f}} =\frac{\parallel p_{\bm{\tau}}\tilde{\bm{f}} \parallel^2}{\parallel \tilde{\bm{f}} \parallel^2}
~~~\mathrm{and}~~~
\beta_{\bm{f}} =\frac{\parallel p_{\tilde{\bm{\sigma}}}\mathcal{F} (\tilde{\bm{f}}) \parallel^2}{\parallel \mathcal{F} (\tilde{\bm{f}}) \parallel^2}.
\end{eqnarray}
Note that the value of $\alpha_{\bm{f}}$ and $\beta_{\bm{f}} $ are real values  in $[0,1]$.

\subsection{Energy Concentration Problem for $\sigma$-Bandlimited Signals}
 \label{S5.1}
In this part, we only consider the energy problem for $\bm{f} \in \mathscr{B}_{\bm{\sigma}}$, i.e., $\beta_{\bm{f}}=1$ and
$\tilde{\bm{f}}\in \tilde{\mathscr{B}}_{\tilde{\bm{\sigma}}}$.
Concretely speaking, given an unit energy $\tilde{\bm{f}}\in \tilde{\mathscr{B}}_{\tilde{\bm{\sigma}}}$,
the energy concentration problem is finding the maximum of $\alpha_{\bm{f}}$, i.e.,
\begin{eqnarray}
\max_{\tilde{\bm{f}}(x,y)\in \tilde{\mathscr{B}}_{\tilde{\bm{\sigma}}}}
  \alpha_{\bm{f}} =\parallel p_{\bm{\tau}}\tilde{\bm{f}} \parallel^2.
\end{eqnarray}
Denote the maximum $\alpha_{\bm{f}}$ as follows
\begin{eqnarray}
\alpha_{max}:=\max_{\tilde{\bm{f}}(x,y)\in \tilde{\mathscr{B}}_{\tilde{\bm{\sigma}}}} \alpha_{\bm{f}}.
\end{eqnarray}

Let $\bm{\tilde{f}}_{\bm{\tau}}(x,y):=p_{\bm{\tau}}(x,y)\bm{\tilde{f}}(x,y)$, we can also reformulate $\alpha_{\bm{f}}$ as follows
\begin{eqnarray}
\alpha_{\bm{f}} =\int_{\R^2}\bm{\tilde{f}}_{\bm{\tau}}(x,y)\overline{\bm{\tilde{f}}(x,y)}dxdy.
\end{eqnarray}
We conclude that the maximum $\alpha_{max}$ can be taken if $\bm{\tilde{f}}_{\bm{\tau}}(x,y)=\mu  \bm{\tilde{f}}(x,y)$.
To derive this fact, the generally cross-correlation function $\rho_{fg}$ of $f$ and $g\in \mathcal{L}^2(\R^2;\C)$ was introduced at first \cite{P1977},
\begin{eqnarray}
\rho_{fg}(s,t):=\int_{\R^2}\overline{f(x,y)}g(s+x,t+y)dxdy.
\end{eqnarray}
Consider the $(s,t)=(0,0)$, we have
$
\rho_{fg}(0,0)=\int_{\R^2}\overline{f(x,y)}g(x,y)dxdy.
$
From the complex-valued Schwarz's inequality,
\begin{eqnarray}
\left| \int_{\R^2}\overline{f(x,y)}g(x,y)dxdy\right|^2\leq \int_{\R^2}\left|f(x,y)\right|^2dxdy\int_{\R^2}\left|g(x,y)\right|^2dxdy.
\end{eqnarray}
the $|\rho_{fg}(0,0)|^2$ takes the maximum value if $f(x,y)=kg(x,y)$, where $k$ is a constant.
Similarly, we can define the cross-correlation function $\rho_{\bm{fg}}$ of $\H$-valued signals $\bm{f}$ and $\bm{g}\in \mathcal{L}^2(\R^2;\H)$ as follows
\begin{eqnarray}
\rho_{\bm{f}\bm{g}}(s,t):=\int_{\R^2}\overline{\bm{f}(x,y)}\bm{g}(s+x,t+y)dxdy.
\end{eqnarray}
Since the quaternionic Schwarz's inequality also holds.
Then to get the maximum value of $\rho_{\bm{f}\bm{g}}(0,0)$,
the relationship between $\bm{f}$ and $\bm{g}$ satisfies  $\bm{f}(x,y)=\gamma\bm{g}(x,y)$, where $\gamma$ is a constant.
Here, we find that
\begin{eqnarray}
\alpha_{\bm{f}}=\rho_{\bm{\tilde{f}}_{\bm{\tau}}\bm{\tilde{f}}}(0,0).
\end{eqnarray}
To achieve the maximum $\alpha_{\bm{f}}$, the two functions should be the same except a constant factor.
For this reason, there exists a constant $\mu$ such that
$\bm{\tilde{f}}_{\bm{\tau}}(x,y)=\mu  \bm{\tilde{f}}(x,y)$.

Let  $\F(\bm{\tilde{f}}_{\bm{\tau}})$ and $\F(\bm{\tilde{f}})$ are the QFT for $\bm{\tilde{f}}_{\bm{\tau}}$ and $\bm{\tilde{f}}$, respectively.
Taking QFT to both sides of the equation $\bm{\tilde{f}}_{\bm{\tau}}(x,y)=\mu  \bm{\tilde{f}}(x,y)$, we have
\begin{eqnarray}
\F(\bm{\tilde{f}}_{\bm{\tau}})(u,v)=\mu \F(\bm{\tilde{f}})(u,v).
\end{eqnarray}
Since $\bm{\tilde{f}}\in\tilde{\mathscr{B}}_{\tilde{\bm{\sigma}}}$, then $\bm{\tilde{f}}_{\bm{\tau}}(x,y)$ is also in $\tilde{\mathscr{B}}_{\tilde{\bm{\sigma}}}$, i.e.,
\begin{eqnarray}\label{Eq.QFTconvl}
\F(\bm{\tilde{f}}_{\bm{\tau}})(u,v)p_{\tilde{\bm{\sigma}}}(u,v)=\mu \F(\bm{\tilde{f}})(u,v).
\end{eqnarray}
From Lemma \ref{Lem.convolution}, taking the inverse QFT to the above equation, we have
\begin{eqnarray}\label{Eq.convl}
\int_{\R^2}\bm{\tilde{f}}_{\bm{\tau}}(s,t)\frac{\sin \sigma (x-s)}{\pi (x-s)}
\frac{\sin \sigma (y-t)}{\pi (y-t)}dsdt=\mu \bm{\tilde{f}}(x,y).
\end{eqnarray}
Substituting $\bm{\tilde{f}}_{\bm{\tau}}(s,t)=p_{\bm{\tau}}(s,t)\bm{\tilde{f}}(s,t)$ to the above equation, we have
\begin{eqnarray}\label{Eq.low}
\int_{\bm{\tau}}\bm{\tilde{f}}(s,t)\frac{\sin \sigma (x-s)}{\pi (x-s)}
\frac{\sin \sigma (y-t)}{\pi (y-t)}dsdt=\mu \bm{\tilde{f}}(x,y),
\end{eqnarray}
which is the low-pass filter form of QPSWFs.

Now we show that $\bm{\sigma}$-bandlimited $\H$-valued signals satisfying
the low-pass filter form Eq. (\ref{Eq.low}) can reach the maximum $\alpha_{max}$.

\begin{Theorem}\label{Th.TheoremExist}
If the eigenvalues of the integral equation
\begin{eqnarray}\label{Eq.Exist}
\int_{\bm{\tau}}\bm{\tilde{f}}(s,t)\frac{\sin \sigma (x-s)}{\pi (x-s)}
\frac{\sin \sigma (y-t)}{\pi (y-t)}dsdt=\mu \bm{\tilde{f}}(x,y),
\end{eqnarray}
have a maximum $\mu$, then $\alpha_{max}=\mu _{max}$.
The eigenfunction corresponding to $\mu_{max}$
is the function such that $\alpha_{max}$ are reached.
\end{Theorem}

\begin{proof}
For any $\bm{\sigma}$-bandlimited signal $\bm{\tilde{f}}$, construct a function $\bm{\tilde{s}}(x,y)$  as follows
\begin{eqnarray}\label{Eq.s}
\bm{\tilde{s}}(x,y):=\int_{\bm{\tau}}\bm{\tilde{f}}(s,t)\frac{\sin \sigma (x-s)}{\pi (x-s)}\frac{\sin \sigma (y-t)}{\pi (y-t)}dsdt.
\end{eqnarray}
Let the QFT of $\bm{\tilde{s}}(x,y)$  as $\F(\bm{\tilde{s}})(u,v)$,
it follows that
\begin{eqnarray*}
\F(\bm{\tilde{s}})(u,v)=
\F(\bm{\tilde{f}}_{\bm{\tau}})(u,v) p_{\tilde{\bm{\sigma}}}(u,v).
\end{eqnarray*}
It means that $\bm{\tilde{s}}\in \tilde{\mathscr{B}}_{\tilde{\bm{\sigma}}}$.

Denote the energy ratio  $\alpha_{\bm{s}}$ for $\bm{\tilde{s}}(x,y)$   in the fixed spatial domain $\bm{\tau}$ as follows
\begin{eqnarray}
\alpha_{\bm{s}}=\frac{1}{E_{\bm{s}}}\int_{\R^2}p_{\bm{\tau}}(x,y)\bm{\tilde{s}}(x,y)\overline{\bm{\tilde{s}}(x,y)}dxdy.
\end{eqnarray}
We conclude that for any $\bm{\tilde{f}}\in \tilde{\mathscr{B}}_{\tilde{\bm{\sigma}}}$, $\alpha_{\bm{f}}$ cannot exceed the $\alpha_{\bm{s}}$.
Direct computations show that the energy of  the signal $\bm{\tilde{s}}(x,y)$ is given as follows
\begin{eqnarray*}
E_{\bm{s}}&=&\int_{\R^2}\bm{\tilde{s}}(x,y)\overline{\bm{\tilde{s}}(x,y)}dxdy\\
&=&\int_{\R^2}\F(\bm{\tilde{s}})(u,v)
\overline{\F(\bm{\tilde{s}})(u,v)}dudv\\
&=&{\bf Sc} \left[\int_{\tilde{\bm{\sigma}}}\F(\bm{\tilde{f}}_{\bm{\tau}})(u,v)
\overline{\F(\bm{\tilde{s}})(u,v)}dudv\right]\\
&=&{\bf Sc} \left[ \int_{\R^2}\bm{\tilde{f}}_{\bm{\tau}}(x,y)\bm{\tilde{s}}(x,y)dxdy \right]\\
&=&{\bf Sc} \left[ \int_{\R^2}p_{\bm{\tau}}(x,y)\bm{\tilde{f}}(x,y)\overline{\bm{\tilde{s}}(x,y)}dxdy\right].
\end{eqnarray*}

On the other hand, we consider that
\begin{eqnarray*}
\alpha_{\bm{f}}E_{\bm{f}}&=&\int_{\R^2}\bm{\tilde{f}}_{\bm{\tau}}(x,y)\overline{\bm{\tilde{f}}(x,y)}dxdy\\
&=&{\bf Sc} \left[ \int_{\R^2}\F(\bm{\tilde{f}}_{\bm{\tau}})(u,v)
\overline{\F(\bm{\tilde{f}})(u,v)}dudv\right]\\
&=&{\bf Sc} \left[ \int_{\tilde{\bm{\sigma}}}\F(\bm{\tilde{s}})(u,v)
\overline{\F(\bm{\tilde{f}})(u,v)}dudv \right].
\end{eqnarray*}
Since
\begin{eqnarray} \label{Eq.1}
\left|\int_{\tilde{\bm{\sigma}}}\F(\bm{\tilde{s}})(u,v)
\overline{\F(\bm{\tilde{f}})(u,v)}dudv\right|^2
\leq \parallel p_{\tilde{\bm{\sigma}}}\mathcal{F} (\tilde{\bm{s}}) \parallel^2\parallel p_{\tilde{\bm{\sigma}}}\mathcal{F} (\tilde{\bm{f}}) \parallel^2,
\end{eqnarray}
and
\begin{eqnarray*}
\left| {\bf Sc} \left[ \int_{\tilde{\bm{\sigma}}}\F(\bm{\tilde{s}})(u,v)
\overline{\F(\bm{\tilde{f}})(u,v)}dudv \right] \right|^2
\leq
\left|\int_{\tilde{\bm{\sigma}}}\F(\bm{\tilde{s}})(u,v)
\overline{\F(\bm{\tilde{f}})(u,v)}dudv\right|^2,
\end{eqnarray*}
simplifying the above three inequalities, we obtain that
\begin{eqnarray*}
(\alpha_{\bm{f}}E_{\bm{f}})^2\leq E_{\bm{s}}E_{\bm{f}},
\end{eqnarray*}
 from which it follows that
\begin{eqnarray*}
 \alpha_{\bm{f}}^2E_{\bm{f}}\leq E_{\bm{s}}.
\end{eqnarray*}
We also have the following result for $E_{\bm{s}}$
\begin{eqnarray}\label{Eq.2}
E_{\bm{s}}^2&=&
\left|{\bf Sc} \left[ \int_{\R^2}p_{\bm{\tau}}(x,y)\bm{\tilde{f}}(x,y)\overline{\bm{\tilde{s}}(x,y)}dxdy\right]\right|^2\\  \nonumber
&\leq&
\left|\int_{\R^2}p_{\bm{\tau}}(x,y)\bm{\tilde{f}}(x,y)\overline{\bm{\tilde{s}}(x,y)}dxdy\right|^2\\  \nonumber
&\leq& \int_{\R^2}p_{\bm{\tau}}(x,y)\bm{\tilde{f}}(x,y)\overline{\bm{\tilde{f}}(x,y)}dxdy
\int_{\R^2}p_{\bm{\tau}}(x,y)\bm{\tilde{s}}(x,y)\overline{\bm{\tilde{s}}(x,y)}dxdy.
\end{eqnarray}
Here, we take $p_{\bm{\tau}}(x,y)$ into two parts, i.e.,  $\Big(\sqrt{p_{\bm{\tau}}(x,y)}\Big)^2$, and use the Schwarz inequality for the above inequality.
Clearly,  $(E_{\bm{s}})^2\leq (\alpha_{\bm{f}}E_{\bm{f}})(\alpha_{\bm{s}}E_{\bm{s}})$, then
\begin{eqnarray*}
E_{\bm{s}}\leq \alpha_{\bm{f}}\alpha_{\bm{s}}E_{\bm{f}}.
\end{eqnarray*}
Summarizing, we have
\begin{eqnarray*}
 \alpha_{\bm{f}}^2E_{\bm{f}}\leq E_{\bm{s}}\leq \alpha_{\bm{f}}\alpha_{\bm{s}}E_{\bm{f}}.
\end{eqnarray*}
That means, for any $\bm{\tilde{f}} \in \tilde{\mathscr{B}}_{\tilde{\bm{\sigma}}}$, $\alpha_{\bm{f}}\leq \alpha_{\bm{s}}$.

If  $\alpha_{\bm{f}}=\alpha_{\bm{s}}$, then  Eq. (\ref{Eq.1}) and Eq. (\ref{Eq.2}) must be equalities.
This  is attained only by setting $\bm{\tilde{s}}(x,y)= \alpha_{\bm{f}}\bm{\tilde{f}}(x,y)$ with
 $E_{\bm{s}}= \alpha_{\bm{f}}^2E_{\bm{f}}$. 
 It means that
 $\bm{\tilde{f}}$ is an eigenfunction of Eq. (\ref{Eq.Exist})
 and $\alpha_{\bm{f}}$ is the corresponding eigenvalue, i.e.,  $\mu=\alpha_{\bm{f}}$.

At last, we will show that $\alpha_{max}=\mu _{max}$, and the eigenfunction corresponding to $\mu_{max}$
is the function such that $\alpha_{max}$ is reached.
By definition of $0\leq\alpha_{\bm{f}}\leq1$, there exists a maximum $\alpha_{f}$ and we denote the
maximum $\alpha_{\bm{f}}$ as $\alpha_{max}$ and the corresponding signal as $\bm{\tilde{f}}_{0}(x,y)$.
As we have shown, the $\alpha_{max}$  corresponding the eigenfunction satisfies $\bm{\tilde{s}}(x,y)= \alpha_{\bm{f}}\bm{\tilde{f}}(x,y)$.
Here, $\bm{\tilde{s}}(x,y)= \alpha_{\bm{f}}\bm{\tilde{f}}(x,y)=\alpha_{\bm{f}}\bm{\tilde{f}}_{0}(x,y)$ corresponds to the maximum eigenvalue of $\mu_{max}$.
Hence, $\mu_{max}\leq\alpha_{max}$.

In order to  prove that $\mu_{max}=\alpha_{max}$, it suffices to show that
$\bm{\tilde{f}}_{0}(x,y)$ is an eigenfunction of the integral equation Eq. (\ref{Eq.Exist}),
or equivalently, that with $\bm{\tilde{S}}_0(x,y)$ defined as $\bm{\tilde{s}}(x,y)$ in Eq. (\ref{Eq.s}) with $\alpha_{\bm{S}_0}=\alpha_{max}$.
Obviously, $\alpha_{\bm{S}}\geq\alpha_{max}$ and
$\alpha_{\bm{S}_0}\leq\alpha_{max}$,  because $\alpha_{max}$ is maximum by assumption.
The proof is complete.
\end{proof}

Theorem \ref{Th.TheoremExist} shows that
for arbitrary unit energy $\bm{\sigma}$-bandlimited $\H$-valued signal associated with QLCTs
 the maximum value of $\alpha_{\bm{f}}$ can be achieved by the QPSWFs.
 In fact, from the symmetry theorem of Fourier theory \cite{P1977},
 there is also a similar  integral equation for time-limited signals,
 which have the maximum $\beta_{\tilde{\bm{f}}}$.
 The prove of this conclusion is similar to Theorem \ref{Th.TheoremExist}.
\begin{Cor}
If the eigenvalues of the integral equation
\begin{eqnarray}\label{Eq.Fptauf}
\int_{\tilde{\bm{\sigma}}}\F(\bm{\tilde{f}})(u,v)\frac{\sin \tau (x-u)}{\pi (x-u)}
\frac{\sin \tau(y-v)}{\pi (y-v)}dudv
=\mu\F(\bm{\tilde{f}})(x,y).
\end{eqnarray}
have a maximum $\mu$, then $\beta_{\tilde{\bm{f}}}$ have a maximum number $\beta_{max}$ and $\beta_{max}=\mu _{max}$.
The eigenfunction corresponding to $\mu_{max}$
is the function such that $\beta_{max}$ are reached.
\end{Cor}
The Eq. (\ref{Eq.Fptauf})  is equivalent to Eq. (\ref{Eq.Exist}) with $u=\frac{\sigma s}{\tau}$ and $v=\frac{\sigma t}{\tau}$.

\subsection{Extremal Properties}
\label{S5.2}
In this section, we will discuss the relationship of $(\alpha_{\bm{f}}, \beta_{\bm{f}})$ in Eq. (\ref{Eq.extremal}) from three cases:
\begin{itemize}
  \item [(1)]  $\bm{f}(x,y)$ is a $\bm{\sigma}$-bandlimited signal  associated with QLCTs.
  \item [(2)]  $\bm{f}(x,y)$ is a $\bm{\tau}$-time-limited signal.
  \item [(3)]  $\bm{f}(x,y)$ is an arbitrary signal.
\end{itemize}

The first case follows form the general theory of the $\bm{f}\in \mathscr{B}_{\bm{\sigma}}$ in Section \ref{S5.1}.
As we have known $\bm{\tilde{f}}$ is in $\tilde{\mathscr{B}}_{\tilde{\bm{\sigma}}}$
when $\bm{f}\in \mathscr{B}_{\bm{\sigma}}$, i.e., $\beta_{\bm{f}}=1$.
From Theorem \ref{Th.TheoremExist},
we know that the maximum $\alpha_{\bm{f}}$ equals the maximum eigenvalue $\mu_0$ in Eq. (\ref{Eq.Exist}).
Using the expansion for the $\bm{\tilde{f}} \in \tilde{\mathscr{B}}_{\tilde{\bm{\sigma}}}$,
$
\bm{\tilde{f}}(x,y)=\Sigma_{n=0}^{\infty}a_n\tilde{\bm{\psi}}_{n}(x,y),
$
where $a_n:=\int_{\R^2}\bm{\tilde{f}}(x,y)\overline{\bm{\psi}_{n}(x,y)}dxdy$. It is clear that
$
\alpha_{\bm{f}}=\int_{\bm{\tau}}\bm{\tilde{f}}(x,y)\overline{\bm{\tilde{f}}(x,y)}dxdy
=\Sigma_{n=0}^{\infty}\mu_n a_n^2\leq\mu_0\Sigma_{n=0}^{\infty}a_n^2=\mu _0.
$
Hence, $\alpha_{\bm{f}}\leq\mu_0$. If  $\alpha_{\bm{f}}=\mu _0$, then $\bm{\tilde{f}}(x,y)=\bm{\tilde{\psi}}_{0}(x,y)$.
If $\alpha_{\bm{f}}<\mu_0$, then we can find a signal $\bm{f}\in \mathscr{B}_{\bm{\sigma}}$  whose energy ratio in spatial domain equals
$\alpha_{\bm{f}}$, and in this case, $\bm{\tilde{f}}(x,y)$ is not unique.

The second case means $\alpha_{\bm{f}}=1$.
From the property of symmetry of the QLCT we conclude that all the properties for signals $\bm{f}\in \mathscr{B}_{\bm{\sigma}}$ have
corresponding time-limited counterparts. Reversing $(x,y)$ and $(u,v)$, we conclude that $\beta_{\bm{f}}\leq\mu_0$.
Specially, if $\beta_{\bm{f}}=\mu _0$, then $\bm{\tilde{f}}(x,y)=\frac{p_{\bm{\tau}}(x,y)\bm{\tilde{\psi}}_{0}(x,y)}{\sqrt{\mu_0}}$.

For the third case, considering arbitrary signals with $\alpha_{\bm{f}}<1$,
we aim to find the maximum $\beta_{\bm{f}}$ and the corresponding signal $\bm{f}(x,y)$.
If $\alpha_{\bm{f}}\leq\mu_0$, as we noted in the case of $\bm{f}\in \mathscr{B}_{\bm{\sigma}}$,
we can find $\tilde{\bm{f}}\in \tilde{\mathscr{B}}_{\tilde{\bm{\sigma}}}$  with energy ration $\alpha_{\bm{f}}$, hence, $\beta_{max}=1$.
Therefore, we only need to consider the  case of $\alpha_{\bm{f}}>\mu_0$.

\begin{Theorem}\label{Th.maxbeta}
The maximum $\beta_{max}$ of $\beta_{\bm{f}}$  must satisfy the following equation
\begin{eqnarray}\label{Eq.maxbeta}
\arccos\sqrt{\beta_{\bm{f}}}+\arccos\sqrt{\alpha_{\bm{f}}}=\arccos\sqrt{\mu_0},
\end{eqnarray}
where $\mu_0$ is the largest eigenvalues of Eq. (\ref{Eq.lowpass})
and the corresponding $\bm{\tilde{f}}$ for the  maximum $\beta_{max}$  is given by
\begin{eqnarray}
\bm{\tilde{f}}(x,y)=\sqrt{\frac{1-\alpha_{\bm{f}}}{1-\mu_0}}p_{\bm{\tau}}(x,y)\bm{\tilde{\psi}}_{0}(x,y)+
\left(\sqrt{\frac{\alpha_{\bm{f}}}{\mu_0}}-\alpha_{\bm{f}}\right)\bm{\tilde{\psi}}_{0}(x,y).
\end{eqnarray}
\end{Theorem}
\begin{proof}
Before giving the proof to Eq. (\ref{Eq.maxbeta}), we first need to present the following fact.
Given a function $\bm{\tilde{f}}$ with spatial projection $p_{\bm{\tau}}\bm{\tilde{f}}$ and
frequency projection $p_{\tilde{\bm{\sigma}}}\F(\bm{\tilde{f}})$,
we construct a new function  as follows
\begin{eqnarray}
\bm{\tilde{f}}_1(x,y):=ap_{\bm{\tau}}\bm{\tilde{f}}(x,y)+b\F^{-1}\Big(p_{\tilde{\bm{\sigma}}}\F(\bm{\tilde{f}})\Big)(x,y),
\end{eqnarray}
where $a$ and $b$ are two constants such that the energy of $\bm{g}(x,y)$ is minimum, where
\begin{eqnarray}
\bm{g}(x,y):=\bm{\tilde{f}}(x,y)-\tilde{\bm{f}}_{1}(x,y).
\end{eqnarray}
Denote  $\alpha_{\bm{f}}$, $\beta_{\bm{f}}$ and $\alpha_{\bm{f}_1}$, $\beta_{\bm{f}_1}$ the energy ratios
for $\bm{\tilde{f}}(x,y)$ and $\tilde{\bm{f}}_{1}(x,y)$ as Eq. (\ref{Eq.extremal}), respectively.
We conclude that $\alpha_{\bm{f}_1}\geq\alpha_{\bm{f}}$, $\beta_{\bm{f}_1}\geq\beta_{\bm{f}}$.

Suppose the energy of $\bm{\tilde{f}}(x,y)$ equals to $1$ and we rewrite  $\alpha_{\bm{f}},~\beta_{\bm{f}}$ as follows
\begin{eqnarray}
\begin{array}{ll}
\alpha_{\bm{f}}~=~\Big\langle p_{\bm{\tau}}\bm{\tilde{f}}, p_{\bm{\tau}}\bm{\tilde{f}}\Big\rangle,\\
\beta_{\bm{f}}~=~\Big\langle p_{\tilde{\bm{\sigma}}}\F(\bm{\tilde{f}}), p_{\tilde{\bm{\sigma}}}\F(\bm{\tilde{f}})\Big\rangle.
\end{array}
\end{eqnarray}
From the orthogonality principle \cite{P1977}, it follows that
\begin{eqnarray}
\Big\langle p_{\bm{\tau}}\bm{\tilde{f}}, \bm{g}\Big\rangle=0,~~~\textrm{and}~~~
\Big\langle \F^{-1}\Big(p_{\tilde{\bm{\sigma}}}\F(\bm{\tilde{f}})\Big), \bm{g}\Big\rangle=0,
\end{eqnarray}
which means $\Big\langle \tilde{\bm{f}}_{1}, \bm{g}\Big\rangle=0$.
Meanwhile, we have $E_{\bm{f}_1}$ of $\tilde{\bm{f}}_{1}$ by
\begin{eqnarray}
E_{\bm{f}_1}:=\Big\langle\tilde{\bm{f}}_{1},\tilde{\bm{f}}_{1} \Big\rangle
=1-E_{\bm{g}}.
\end{eqnarray}
Now we denote two energy for the  projection of $\bm{g}(x,y)$ as follows
\begin{eqnarray}
E_{p_{\bm{\tau}}\bm{g}}:=\Big\langle p_{\bm{\tau}}\bm{g},p_{\bm{\tau}}\bm{g}\Big\rangle,~~~\textrm{and}~~~
E_{p_{\tilde{\bm{\sigma}}}\F(\bm{g})}
:=\Big\langle p_{\tilde{\bm{\sigma}}}\F(\bm{g}),p_{\tilde{\bm{\sigma}}}\F(\bm{g})\Big\rangle.
\end{eqnarray}
The $E_{p_{\bm{\tau}}f}$ and $E_{p_{\tilde{\bm{\sigma}}}\F(\bm{f})}$ will be simply written
as $E_{\bm{\tau}}$ and $E_{\tilde{\bm{\sigma}}}$ in the following, respectively.
Since $\tilde{\bm{f}}_{1}(x,y)=\bm{\tilde{f}}(x,y)-\bm{g}(x,y)$,
we have
\begin{eqnarray}
\begin{array}{ll}
\Big\langle p_{\bm{\tau}}\tilde{\bm{f}}_{1},p_{\bm{\tau}}\tilde{\bm{f}}_{1}\Big\rangle
=\alpha_{\bm{f}_1}E_{\bm{f}_1}=\alpha_{\bm{f}}+E_{\bm{\tau}}, \\
\Big\langle p_{\tilde{\bm{\sigma}}}\F(\tilde{\bm{f}}_{1}),p_{\tilde{\bm{\sigma}}}\F(\tilde{\bm{f}}_{1})\Big\rangle
=\beta_{\bm{f}_1}E_{\bm{f}_1}=\beta_{\bm{f}}+E_{\tilde{\bm{\sigma}}}.
\end{array}
\end{eqnarray}
Therefore,  $\alpha_{\bm{f}_1}\geq\alpha_{\bm{f}}$ and $\beta_{\bm{f}_1}\geq\beta_{\bm{f}}$.
That means, in order to get the maximum $\beta_{\bm{f}}$, we can formula a function as follows
\begin{eqnarray}\label{Eq.combi}
\bm{\tilde{f}}=ap_{\bm{\tau}}\bm{\tilde{f}}+b\F^{-1}\Big(p_{\tilde{\bm{\sigma}}}\F(\bm{\tilde{f}})\Big).
\end{eqnarray}
Taking QFT to both sides for Eq. (\ref{Eq.combi}) and then taking frequency projection, we have
\begin{eqnarray}
\F\left(\bm{\tilde{f}}\right)p_{\tilde{\bm{\sigma}}}
=
a p_{\tilde{\bm{\sigma}}} \F\left(\bm{\tilde{f}}\right)*\left(\frac{\sin (\tau u)}{\pi u}\frac{\sin (\tau v)}{\pi v}\right)
+bp_{\tilde{\bm{\sigma}}}\F(\bm{\tilde{f}}).
\end{eqnarray}
Rearranging this formula, we obtain that
\begin{eqnarray}
(1-b)\F\left(\bm{\tilde{f}}\right)p_{\tilde{\bm{\sigma}}}=
a p_{\tilde{\bm{\sigma}}}\left[\F\Big(\bm{\tilde{f}}\Big)*\Big(\frac{\sin (\tau u)}{\pi u}\frac{\sin (\tau v)}{\pi v}\Big)\right].
\end{eqnarray}
Taking inverse QFT to the above equation, we have
\begin{eqnarray}\label{eq.inversepafts5}
&&\frac{1-b}{a}\F^{-1}\left(\F\left(\bm{\tilde{f}}\right)p_{\tilde{\bm{\sigma}}}\right)\\ \nonumber
&=&\F^{-1}\left(\F\left(\bm{\tilde{f}}\right)*\left(\frac{\sin (\tau u)}{\pi u}\frac{\sin (\tau v)}{\pi v}\right)\right)
*\left(\frac{\sin (\sigma x)}{\pi x}\frac{\sin (\sigma y)}{\pi y}\right).
\end{eqnarray}
On the other hand, taking the spatial projection to Eq. (\ref{Eq.combi}), we get
 \begin{eqnarray}
p_{\bm{\tau}}(x,y)\bm{\tilde{f}}(x,y)
=ap_{\bm{\tau}}(x,y)\bm{\tilde{f}}(x,y)+b\F^{-1}\left(p_{\tilde{\bm{\sigma}}}\F(\bm{\tilde{f}})\right)(x,y)p_{\bm{\tau}}(x,y).
\end{eqnarray}
Rearranging this equation, it becomes
 \begin{eqnarray}\label{eq.fandF5}
(1-a)p_{\bm{\tau}}(x,y)\bm{\tilde{f}}(x,y)=bp_{\bm{\tau}}(x,y) \F^{-1}\left(p_{\tilde{\bm{\sigma}}}\F(\bm{\tilde{f}})\right)(x,y).
\end{eqnarray}
Taking QFT on both sides to the above equation, it follows that
 \begin{eqnarray}\label{eq.pafts5}
(1-a)\F\left(\bm{\tilde{f}}\right)*\left(\frac{\sin (\tau u)}{\pi u}\frac{\sin (\tau v)}{\pi v}\right)
=b\left(p_{\tilde{\bm{\sigma}}}\F(\bm{\tilde{f}})\right)
*\left(\frac{\sin (\tau u)}{\pi u}\frac{\sin (\tau v)}{\pi v}\right).
\end{eqnarray}
Applying Eq. (\ref{eq.inversepafts5}) and Eq. (\ref{eq.pafts5}), we have
 \begin{eqnarray*}
&&\frac{1-b}{a}\mathcal{F}^{-1}\left(\mathcal{F}\left(\bm{\tilde{f}}\right)p_{\tilde{\bm{\sigma}}}\right)\\
&=&
\F^{-1}\left(\F\left(\bm{\tilde{f}}\right)*\left(\frac{\sin (\tau u)}{\pi u}\frac{\sin (\tau v)}{\pi v}\right)\right)
*\left(\frac{\sin (\sigma x)}{\pi x}\frac{\sin (\sigma y)}{\pi y}\right)\\
&=& \frac{b}{1-a} \F^{-1}\left(\left(p_{\tilde{\bm{\sigma}}}\F(\bm{\tilde{f}})\right)
*\left(\frac{\sin (\tau u)}{\pi u}\frac{\sin (\tau v)}{\pi v}\right)\right)
*\left(\frac{\sin (\sigma x)}{\pi x}\frac{\sin (\sigma y)}{\pi y}\right)\\
&=&\frac{b}{1-a} \F^{-1}\left(p_{\tilde{\bm{\sigma}}}\F(\bm{\tilde{f}})\right)
p_{\bm{\tau}}(x,y)*\left(\frac{\sin (\sigma x)}{\pi x}\frac{\sin (\sigma y)}{\pi y}\right) .
\end{eqnarray*}
Simplifying the above equality, we obtain that
\begin{eqnarray*}
\frac{(1-a)(1-b)}{ab}\F^{-1}\left(\F\left(\bm{\tilde{f}}\right)p_{\tilde{\bm{\sigma}}}\right)
=p_{\bm{\tau}}(x,y)
\F^{-1}\left(\F\left(\bm{\tilde{f}}\right)p_{\tilde{\bm{\sigma}}}\right)
*\left(\frac{\sin \sigma x}{\pi x}\frac{\sin \sigma y}{\pi y}\right).
\end{eqnarray*}
From above equality, we find that
$\F^{-1}\Big(\F\left(\bm{\tilde{f}}\right)p_{\tilde{\bm{\sigma}}}\Big)$
is one of QPSWFs for Eq. (\ref{Eq.lowpass}) and the corresponding eigenvalue is $\frac{(1-a)(1-b)}{ab}$.
By the relationship between $\bm{\tilde{f}}$ and
$\F^{-1}\Big(\F\left(\bm{\tilde{f}}\right)p_{\tilde{\bm{\sigma}}}\Big)$ in Eq. (\ref{eq.fandF5}),
we conclude that   $\bm{\tilde{f}}$  in Eq. (\ref{Eq.combi}) can be rewritten as
\begin{eqnarray}
\bm{\tilde{f}}(x,y)=A\bm{\tilde{\psi}}(x,y)+Bp_{\bm{\tau}}(x,y)\bm{\tilde{\psi}}(x,y).
\end{eqnarray}
Now, we compute the inner product of the above equation with $\bm{\tilde{f}}$ and $p_{\bm{\tau}}\bm{\tilde{f}}$ respectively.
Since $E_{\bm{f}}=1$ for $\bm{\tilde{f}}$,  we have
\begin{eqnarray}
1&=& A^2+ \mu B^2+2AB\mu,\\ \nonumber
\alpha_{\bm{f}}&=&(A+B)^2\mu.
\end{eqnarray}
Then we have  $A=\sqrt{\frac{1-\alpha_{\bm{f}}}{1-\mu}}$ and $B=\sqrt{\frac{\alpha_{\bm{f}}}{\mu}}-\sqrt{\frac{1-\alpha_{\bm{f}}}{1-\mu}}$.
It follows that
\begin{eqnarray}
\beta_{\bm{f}}=\langle p_{\tilde{\bm{\sigma}}}\F(\bm{\tilde{\psi})}, p_{\tilde{\bm{\sigma}}}\F
(\bm{\tilde{\psi}})\rangle
=(A+B\mu)^2.
\end{eqnarray}
With $\sqrt{\alpha_{\bm{f}}}=\cos\theta$ and $\sqrt{\mu}=\cos\theta_1$, the parameters become
$A=\frac{\sin\theta_1}{\sin\theta}$ and $B=\frac{\cos\theta_1}{\cos\theta}-\frac{\sin\theta_1}{\sin\theta}$.
That means
\begin{eqnarray}
\sqrt{\beta_{\bm{f}}}=\frac{\sin\theta}{\sin\theta_1}+\Big(\frac{\cos\theta}{\cos\theta_1}-\frac{\sin\theta}{\sin\theta_1}\Big)\cos^2\theta_1
=\cos(\theta-\theta_1),
\end{eqnarray}
from which it follows that
\begin{eqnarray}
\arccos\sqrt{\beta_{\bm{f}}}+\arccos\sqrt{\alpha_{\bm{f}}}=\arccos\sqrt{\mu}.
\end{eqnarray}
In order to get the maximal $\beta_{\bm{f}}$, we must take the largest $\mu=\mu_0$.
The corresponding function is
\begin{eqnarray}
\bm{\tilde{f}}(x,y)=\sqrt{\frac{1-\alpha_{\bm{f}}}{1-\mu_0}}\bm{\tilde{\psi}}_{0}(x,y)
+ (\sqrt{\frac{\alpha_{\bm{f}}}{\mu_0}}-\sqrt{\frac{1-\alpha_{\bm{f}}}{1-\mu_0}})p_{\bm{\tau}}\bm{\tilde{\psi}}_{0}(x,y).
\end{eqnarray}
The proof is complete.
\end{proof}

Until now we have discussed all the relationships of $(\alpha_{\bm{f}}, \beta_{\bm{f}})$, as well as
the signals to reach the maximum value of  $\beta_{\bm{f}}$ for different conditions of $\alpha_{\bm{f}}$.

\begin{figure}[!h]
  \centering
    \includegraphics[width=15cm]{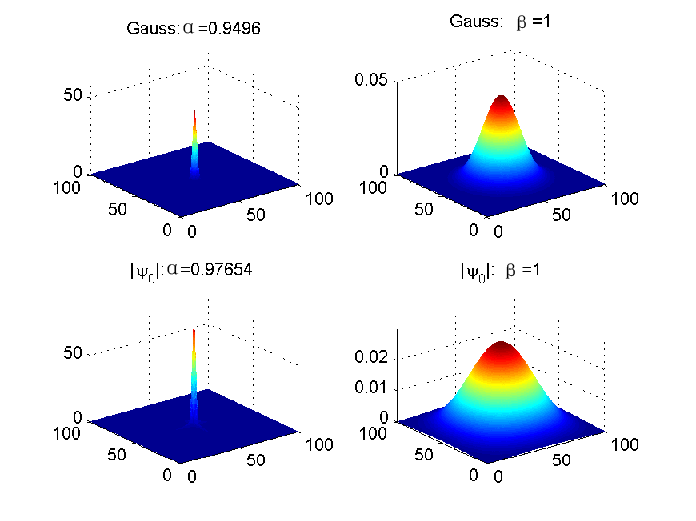}
  \caption{$\bm{\sigma}$-bandlimited $\bm{g}(x,y)$  associated with QFT and the modulus of $\bm{\psi}_{0}(x,y)$ in  time and QFT-frequency domains. }
 \label{fig.beta1}
\end{figure}

\begin{figure}[!h]
  \centering
    \includegraphics[width=15cm]{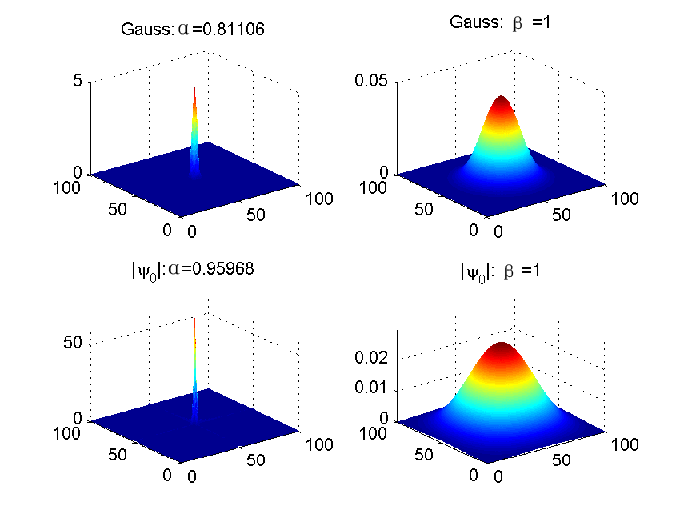}
  \caption{$\bm{\sigma}$-bandlimited $\bm{g}(x,y)$   associated with QLCTs and the modulus of $\psi_{0}(x,y)$ in  time and QLCT-frequency domains with
  $a_1=a_2=0,~b_1=b_2=1,~c_1=c_2=-1,~d_1=d_2=0.1$.  }
 \label{fig.beta2}
\end{figure}

\begin{Exam}
Now we give some comparison examples to intuitively illustrate the concentration levels of QPSWFs  associated with QLCTs.
The widely used Gaussian function will be compared with QPSWFs.
In Theorem \ref{Th.TheoremExist}, we have shown that QPSWFs are the most energy concentred $\bm{\sigma}$-bandlimited signals.

Now, a $\bm{\sigma}$-bandlimited Gaussian function is constructed at first.
Consider the truncated Gaussian function $\bm{g}(x,y)$  in QLCTs-frequency domain as follows
\begin{eqnarray}
G(u,v)=\frac{p_{\bm{\sigma}}(u,v) e^{-(u^2+v^2)}}{\parallel p_{\bm{\sigma}}e^{-(u^2+v^2)} \parallel},
\end{eqnarray}
where $G(u,v)$ is the QLCT of  $\bm{g}(x,y)$.
Obviously, $G(u,v)$ has unit energy.
This $\bm{\sigma}$-bandlimited Gaussian function $\bm{g}(x,y)$ in spatial domain becomes
\begin{eqnarray}
\bm{g}(x,y)=\frac{1}{\parallel p_{\bm{\sigma}}e^{-(u^2+v^2)} \parallel} \L^{-1}\Big( p_{\bm{\sigma}}(u,v) e^{-(u^2+v^2)}\Big).
\end{eqnarray}
As for the QPSWFs, by means of the classical one-dimensional PSWFs of zero order we now construct a special QPSWF as follows
\begin{eqnarray}
\bm{\psi}_{0}(x,y)=\frac{\varphi_0(x)\varphi_0(y)}{\parallel \varphi_0(x)\varphi_0(y) \parallel},
\end{eqnarray}
where $\varphi_0$ is the first one-dimensional zero order PSWF.
Here, we construct the QPSWF under the condition of $c=1$.
The QLCTs for the QPSWF becomes
\begin{eqnarray}
\L\left(\bm{\psi}_{0}\right)(u,v)=\frac{1}{\parallel \varphi_0(x)\varphi_0(y) \parallel}   \L(\varphi_0(x)\varphi_0(y)).
\end{eqnarray}
For both of the $\bm{\sigma}$-bandlimited signals above, the energy ratios $\beta$ equal to $1$ in QLCT-frequency domain.
The energy ratio pair in spatial and frequency in the comparison is noted as $(\alpha, \beta):=(\alpha_{\bm{f}},\beta_{\bm{f}})$.

\begin{figure}[!h]
  \centering
    \includegraphics[width=15cm]{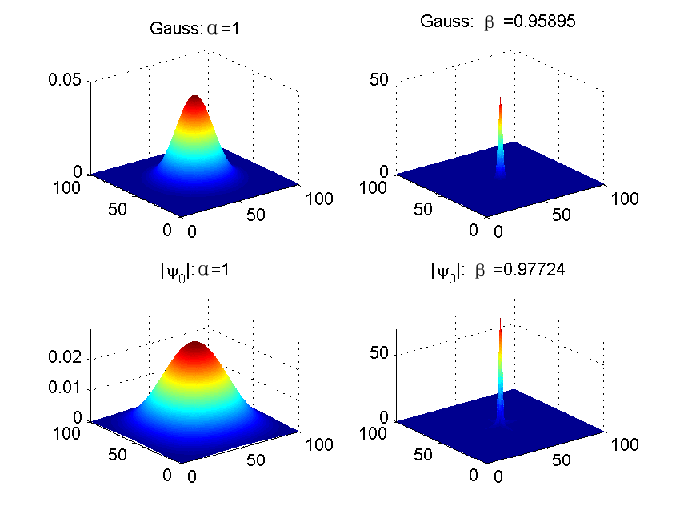}
  \caption{$\bm{\tau}$-time-limited $\bm{g}(x,y)$  associated with QFT and the modulus of $\bm{\psi}_{0}(x,y)$ in  time and QFT-frequency domains. }
 \label{fig.xi1}
\end{figure}
\begin{figure}[!h]
  \centering
    \includegraphics[width=15cm]{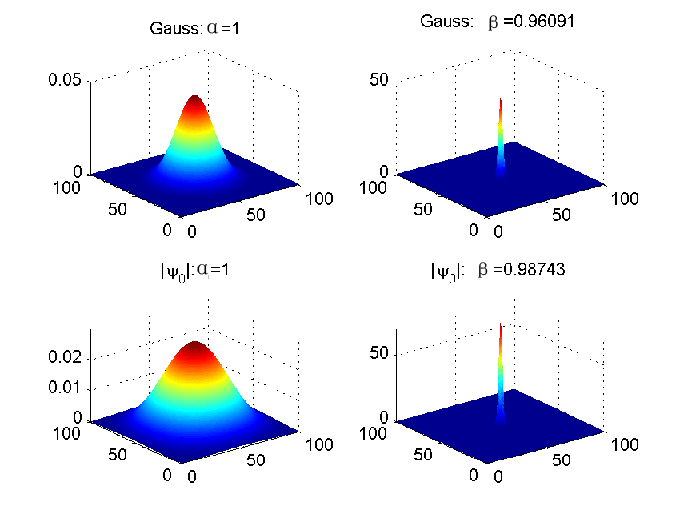}
  \caption{$\bm{\sigma}$-time-limited $\bm{g}(x,y)$  associated with QLCT and the modulus of $\bm{\psi}_{0}(x,y)$ in  time and QLCT-frequency domains with
  $a_1=a_2=0.3,~b_1=b_2=1,~c_1=c_2=-1,~d_1=d_2=0$.  }
 \label{fig.xi2}
\end{figure}

In Fig. \ref{fig.beta1} and Fig. \ref{fig.beta2},
we will show two pairs of the energy ratios $\alpha$ for $\bm{g}(x,y)$ and $\bm{\psi}_{0}(x,y)$ in spatial domain
 associated with QLCT with two kinds of different parameter matrices.
In Fig. \ref{fig.beta1} we set the parameter matrices of QLCT
$A_i=\left(\begin{array}{cc} 0& 1\\-1&0\end{array}\right)$, $i=1,2$,
which is already a QFT. In this case, the energy ratios $\alpha$ for $\bm{g}(x,y)$ and $\bm{\psi}_{0}(x,y)$  are very close.
However, in Fig. \ref{fig.beta2} we set the parameter matrices of QLCT
$A_i=\left(\begin{array}{cc} 0.3& 1\\-1&0\end{array}\right)$, $i=1,2$.
In this case, the energy ratio  $\alpha$ for $\bm{g}(x,y)$ is $0.81106$ and $\alpha$ for $\bm{\psi}_{0}(x,y)$  is $0.95968$.
In fact, we just change the parameters $a_i$, $i=1,2$ from $0$ to $0.3$.
That means, for QPSWFs the energy is more concentred then truncated  Gaussian function.

As for the $\bm{\tau}$ time-limited function, there are  the similar results like  $\bm{\sigma}$-bandlimited cases.
We also list two pairs of the energy ratios $\beta$ for $\bm{g}(x,y)$ and $\bm{\psi}_{0}(x,y)$
in QLCT-frequency domains  in Fig. \ref{fig.xi1} and Fig. \ref{fig.xi2}.
In Fig. \ref{fig.xi1} we also set the parameter matrices of QLCT to be the QFT.
The parameter matrices of QLCT in  Fig. \ref{fig.xi2} is the same as that in  Fig. \ref{fig.beta2}.
In this two pair cases, you may see the energy ratios $\beta$ for $\bm{g}(x,y)$ and $\bm{\psi}_{0}(x,y)$  are very close.
But one more thing different from Fig. \ref{fig.beta1} and Fig. \ref{fig.beta2} is that
 the energy ratios $\beta$ for $\bm{g}(x,y)$ and $\bm{\psi}_{0}(x,y)$  associated with QFT is smaller than
 the energy ratios $\beta$ for $\bm{g}(x,y)$ and $\bm{\psi}_{0}(x,y)$  associated with the second  parameter matrices.
 That means, the parameter matrices of QLCT is vary important.
 In some sense, for specific conditions the results for QLCT will be better than QFT.
\end{Exam}

\section{Conclusion}
\label{S6}

This paper presented a new generalization of PSWFs, namely QPSWFs, which are the optimal
$\H$-valued signals for the energy concentration problem associated with the QLCTs.
We developed the definition of the QPSWFs  associated with QLCTs and established various properties of them.
In order to find the energy distribution of $(\alpha_{\bm{f}}, \beta_{\bm{f}})$ for any $\H$-valued signals,
we not only derive the Parseval identity  associated with (two-sided) QLCTs, but also show that the
maximum $\alpha_{\bm{f}}$ for $\bm{\sigma}$-bandlimited signals associated with QLCTs in a fixed spatial domain must be QPSWFs.

\section*{Acknowledgments}
 The authors acknowledges financial support from the National Natural Science Foundation of China under Grant
 (No. 11401606,11501015), University of Macau (No. MYRG2015-00058-FST and No. MYRG099(Y1-L2)-FST13-KKI)
 and the Macao Science and Technology Development Fund (No. FDCT/094/2011/A and No. FDCT/099/2012/A3).

\end{document}